\documentclass{amsart}
\usepackage{amssymb,amscd}

\newtheorem{theorem}{Theorem}[part]
\newtheorem{proposition}[theorem]{Proposition}
\newtheorem{corollary}[theorem]{Corollary}
\newtheorem{lemma}[theorem]{Lemma}

\theoremstyle{definition}
\newtheorem{definition}[theorem]{Definition}
\newtheorem{example}[theorem]{Example}
\newtheorem{construction}[theorem]{Construction}
\theoremstyle{remark}

\newtheorem*{remark}{Remark}

\numberwithin{equation}{part}
\numberwithin{figure}{part}
\numberwithin{section}{part}

\def\C{\mathbb C}
\def\D{\mathbb D}
\def\I{\mathbb I}

\def\R{\mathbb R}
\def\T{\mathbb T}
\def\Z{\mathbb Z}

\def\sK{\mathcal K}

\def\phi{\varphi}

\newcommand{\mb}[1]{{\textbf {\textit#1}}}

\renewcommand{\ge}{\geqslant}
\renewcommand{\le}{\leqslant}

\newcommand{\llra}{\relbar\joinrel\hspace{-1pt}\longrightarrow}

\newcommand{\id}{\mathrm{id}}

\newcommand{\bideg}{\mathop{\rm bideg}}
\newcommand{\mdeg}{\mathop{\rm mdeg}}

\newcommand{\Ker}{\mathop{\rm Ker}}

\newcommand{\Tor}{\mathop{\rm Tor}\nolimits}

\def\conv{\mathop{\mathrm{conv}}}

\newcommand{\zk}{\mathcal Z_{\mathcal K}}

\newcommand{\zp}{\mathcal Z_P}

\begin{document}
\title[Moment-angle manifolds and complexes]{Moment-angle manifolds and complexes.
\\ Lecture notes KAIST'2010}

\author{Taras Panov}
\address{Department of Mathematics and Mechanics, Moscow
State University, Leninskie Gory, 119991 Moscow, Russia\newline
\emph{and}
\newline Institute for Theoretical and Experimental Physics, Moscow
117259, Russia}
\email{tpanov@mech.math.msu.su}

\thanks{The author was supported by the Russian Foundation for
Basic Research, grants no.~10-01-92102-JF and~09-01-00142, and the
State Programme for the Support of Leading Scientific Schools,
grant no.~5413.2010.1.}


\begin{abstract}
These are notes of the lectures given during the Toric Topology
Workshop at the Korea Advanced Institute of Science and Technology
in February 2010. We describe several approaches to moment-angle
manifolds and complexes, including the intersections of quadrics,
complements of subspace arrangements and level sets of moment
maps. We overview the known results on the topology of
moment-angle complexes, including the description of their
cohomology rings, as well as the homotopy and diffeomorphism types
in some particular cases. We also discuss complex-analytic
structures on moment-angle manifolds and methods for calculating
invariants of these structures.
\end{abstract}

\maketitle

\vspace{-1\baselineskip}

\tableofcontents
\newpage


\def\partname{}
\renewcommand\thepart{}
\part*{Preface}
Moment-angle manifolds and complexes are one of the key players in
\emph{toric topology}, a new and actively developing field on the
borders of equivariant topology, symplectic and algebraic
geometry, and combinatorics. As it was remarkably observed by
Bosio and Meersseman in~\cite{bo-me06}, polytopal moment-angle
manifolds admit complex-analytic structures as LVM manifolds,
therefore creating a new link from toric topology to complex
geometry. The general idea behind these lectures was to broaden
this link and whenever possible unify the complex-analytic
approach to moment-angle manifolds via the LVM-manifolds and the
combinatorial topological approach of~\cite{bu-pa02} via the
moment-angle complexes and coordinate subspace arrangements.

In Lecture I we review the convex-geometrical construction of the
moment-angle manifold $\zp$ corresponding to a simple
polytope~$P$, following~\cite{bu-pa02} and~\cite{b-p-r07}. This
manifold embeds into $\C^m$ with trivial normal bundle, and an
explicit framing is specified by writing the embedded submanifold
as a nondegenerate intersection of quadratic surfaces. This chimes
with the approach of~\cite{bo-me06}, who considered nondegenerate
(or transverse) intersections of several homogeneous real quadrics
with a unit sphere in~$\C^m$; these manifolds were referred to as
\emph{links}. A result of~\cite{bo-me06} shows that the class of
links coincides with the class of polytopal moment-angle manifolds
$\zp$ if we allow redundant inequalities in the presentation of
the polytope~$P$. The two approaches to define the same manifolds
are connected by a sort of duality related to the Gale duality in
convex geometry. This duality was already pretty well understood
and described in~\cite{bo-me06} from the point of view of links of
quadrics (`from quadrics to polytopes'). We take time to go the
other way (`from polytopes to quadrics') thoroughly, by starting
from a moment-angle manifold $\zp$ and eventually coming to
different presentations of it by quadrics, and understanding the
underlying convex geometry on the way.

In Lecture II we review several other approaches to moment-angle
manifolds: the original approach of
Davis--Januszkiewicz~\cite{da-ja91} where $\zp$ was defined as a
certain identification space; the approach of~\cite{bu-pa02} via
the \emph{moment-angle complexes} $\zk$ corresponding to
simplicial complexes~$\sK$ (which also leads to a much wider class
of \emph{non polytopal} moment-angle manifolds corresponding to
sphere triangulations~$\sK$); coordinate subspace arrangements and
their complements; and the symplectic approach where $\zp$ appears
as the level set for the moment map used in the construction of
\emph{Hamiltonian toric manifolds} via symplectic reduction. We
omit the homotopy theoretic approach to moment-angle complexes
because of the geometric nature of these lectures; the reader is
referred to~\cite[Ch.~6]{bu-pa02} for the basics (including the
homotopy fibre interpretation of~$\zk$),
and~\cite{gr-th07},~\cite{b-b-c-g10} for more recent developments.

Lecture III is dedicated to the topology of moment-angle manifolds
and complexes. The cohomology ring of a moment-angle complex $\zk$
is isomorphic to the Tor-algebra of the \emph{Stanley--Reisner
face ring} of $\sK$ and may be described as the cohomology of a
Koszul-like dga with an explicit differential. This result was
first proved in~\cite{bu-pa99}; here we give a sketch of a more
recent proof from~\cite{b-b-p04}. We also derive several
corollaries of the calculation of $H^*(\zk)$, including a
description in terms of full subcomplexes of~$\sK$; the
corresponding formula for the Tor-groups of the face ring is known
in the combinatorial commutative algebra as the \emph{Hochster
theorem}. It is quite clear from any of these descriptions of
$H^*(\zk)$ that there is no hope for a reasonable topological
classification of moment-angle manifolds. However, there are two
methods of quite different origin and flavour which may be used to
describe the topology of $\zp$ and $\zk$ explicitly for some
particular polytopes~$P$ and complexes~$\sK$. The first method is
geometric and may be applied to polytopal moment-angle manifolds:
it is based on the interpretation of $\zp$ as an intersection of
quadrics and uses the surgery theory. The topology of
intersections of quadrics is well understood when the number of
quadrics is small (this corresponds to polytopes $P$ with few
facets; the case of three quadrics was done in~\cite{lope89}), or
when the polytope $P$ is of special type, e.g. obtained by
iteratively cutting vertices off a
simplex~\cite{bo-me06},~\cite{gi-lo09}. The second method is
homotopical; different homotopy-theoretical interpretations of the
moment-angle complex $\zk$ from~\cite{bu-pa02} may be used to
identify its homotopy type as that of a wedge of spheres for some
special series of simplicial complexes~$\sK$~\cite{gr-th07}.

Complex-analytic aspects of the theory of moment-angle complexes
are the subject of Lecture IV. Nondegenerate intersections of
quadrics were studied in holomorphic dynamics as the transverse
sets to certain complex foliations. This study led to a discovery
of a new class of compact non-K\"ahler complex-analytic manifolds
in the work of Lopez de Medrano and Verjovsky~\cite{lo-ve97} and
Meersseman~\cite{meer00}, now known as the \emph{LVM-manifolds}.
As we already mentioned above, Bosio and Meersseman~\cite{bo-me06}
were first to observe that the smooth manifolds underlying a large
class of LVM-manifolds are exactly the polytopal moment-angle
manifolds. It therefore became clear that the moment-angle
manifolds $\zp$ admit non-K\"ahler complex structures generalising
the families of Hopf and Calabi--Eckmann manifolds. We review the
construction of LVM-manifolds and give an argument
of~\cite{bo-me06} providing~$\zp$ with a complex structure of an
LVM manifold. In the final section we apply the spectral sequence
of Borel to analyse the Dolbeault cohomology and Hodge numbers of
these complex structures.

I wish to thank the Korea Advanced Institute of Science and
Technology (KAIST) and especially Dong Youp Suh for organising a
Toric Topology Workshop in February 2010, where these lectures
were delivered, and providing excellent work conditions.

\newpage

\setcounter{part}0
\def\partname{Lecture}
\renewcommand\thepart{\Roman{part}}

\part{Moment-angle manifolds from polytopes}

\section{From polytopes to quadrics}
Let $\R^n$ be a Euclidean space with scalar product
$\langle\;,\:\rangle$. We consider convex polyhedrons defined as
intersections of $m$ halfspaces:
\begin{equation}\label{ptope}
  P=\bigl\{\mb x\in\R^n\colon\langle\mb a_i,\mb x\rangle+b_i\ge0\quad\text{for }
  i=1,\ldots,m\bigr\},
\end{equation}
where $\mb a_i\in\R^n$ and $b_i\in\R$. Assume that the hyperplanes
defined by the equations $\langle\mb a_i,\mb x\rangle+b_i=0$ are
in general position, i.e. at most $n$ of them meet at a single
point. Assume further that $\dim P=n$ and $P$ is bounded (which
implies that $m>n$). Then $P$ is an $n$-dimensional \emph{simple
polytope}. Set
\[
  F_i=\bigl\{\mb x\in P\colon\langle\mb a_i,\mb x\rangle+b_i=0\bigr\}.
\]
Since the hyperplanes are in general position, each $F_i$ is
either empty or a \emph{facet} (an $(n-1)$-dimensional face)
of~$P$. If $F_i$ is empty, then the $i$th inequality
in~\eqref{ptope} is \emph{redundant}; removing it does not change
the set~$P$.

Let $A_P$ be the $m\times n$ matrix of row vectors $\mb a_i$, and
$\mb b_P$ be the column vector of scalars $b_i\in\R^n$. Then we
can write~\eqref{ptope} as
\[
  P=\bigl\{\mb x\in\R^n\colon A_P\mb x+\mb b_P\ge\mathbf 0\},
\]
and consider the affine map
\[
  i_P\colon \R^n\to\R^m,\quad i_P(\mb x)=A_P\mb x+\mb b_P.
\]
It embeds $P$ into
\[
  \R^m_\ge=\{\mb y\in\R^m\colon y_i\ge0\quad\text{for }
  1\le i\le m\}.
\]

We identify $\C^m$ (as a real vector space) with $\R^{2m}$ using
the map
\begin{equation}\label{crid}
  \mb z=(z_1,\ldots,z_m)\mapsto(x_1,y_1,\ldots,x_m,y_m),
\end{equation}
where $z_k=x_k+iy_k$ for $k=1,\ldots,m$.

We define the space $\mathcal Z_P$ from the commutative diagram
\begin{equation}\label{cdiz}
\begin{CD}
  \mathcal Z_P @>i_Z>>\C^m\\
  @VVV\hspace{-0.2em} @VV\mu V @.\\
  P @>i_P>> \R^m_\ge
\end{CD}
\end{equation}
where $\mu(z_1,\ldots,z_m)=(|z_1|^2,\ldots,|z_m|^2)$. The latter
map may be thought of as the quotient map for the coordinatewise
action of the standard torus
\[
  \T^m=\{\mb z\in\C^m\colon|z_i|=1\quad\text{for }1\le i\le m\}
\]
on~$\C^m$. Therefore, $\T^m$ acts on $\zp$ with quotient $P$,
and $i_Z$ is a $\T^m$-equivariant embedding.

The image of $\R^n$ under $i_P$ is an $n$-dimensional affine plane
in $\R^m$, which can be written as
\begin{equation}\label{iprn}
\begin{aligned}
  i_P(\R^n)&=\{\mb y\in\R^m\colon\mb y=A_P\mb x+\mb b_P\quad
  \text{for some }\mb x\in\R^n\}\\
  &=\{\mb y\in\R^m\colon\varGamma\mb y=\varGamma\mb b_P\},
\end{aligned}
\end{equation}
where $\varGamma=(\gamma_{jk})$ is an $(m-n)\times m$ matrix whose
rows form a basis of linear relations between the vectors~${\mb
a}_i$. That is, $\varGamma$ is of full rank and satisfies the
identity $\varGamma A_P=0$. Then we obtain from~\eqref{cdiz} that
$\zp$ embeds into $\C^m$ as the set of common zeros of $m-n$ real
quadratic equations:
\begin{equation}\label{zpqua}
  i_Z(\zp)=\Bigl\{\mb z\in\C^m\colon\sum_{k=1}^m\gamma_{jk}|z_k|^2=
  \sum_{k=1}^m\gamma_{jk}b_k,\;\text{ for }1\le j\le m-n\Bigr\}.
\end{equation}

The following properties of $\zp$ easily follow from its
construction.

\begin{proposition}\label{easyzp}
\begin{itemize}
\item[(a)] Given a point $z\in\zp$, the $i$th coordinate of $i_Z(z)\in\C^m$ vanishes if and only
if $z$ projects onto a point $\mb x\in P$ such that $\mb x\in
F_i$.

\item[(b)] Adding a redundant inequality to~\eqref{ptope} results in multiplying $\zp$ by a circle.
\end{itemize}
\end{proposition}

\begin{theorem}[{\cite[Cor.~3.9]{bu-pa02}}]\label{zpsmooth}
$\zp$ is a smooth manifold of
dimension $m+n$. Moreover, the embedding $i_Z\colon\zp\to\C^m$ has
$\T^m$-equivariantly trivial normal bundle; a $\T^m$-framing is
determined by any choice of matrix $\varGamma$ in~\eqref{iprn}.
\end{theorem}
\begin{proof}
All assertions will follow from the fact that the intersection of
quadrics~\eqref{zpqua} defining $i_Z(\zp)$ is nondegenerate
(transverse). For simplicity we identify $\zp$ with its image
$i_Z(\zp)\subset\C^m$. The gradients of the $m-n$ quadratic forms
in~\eqref{zpqua} at a point $\mb z=(x_1,y_1,\ldots,x_m,y_m)\in\zp$
are
\begin{equation}\label{grve}
  2\left(\gamma_{j1}x_1,\,\gamma_{j1}y_1,\,\dots,\,\gamma_{jm}x_m,\,\gamma_{jm}y_m\right),\quad
  1\le j\le m-n.
\end{equation}
These gradients form the rows of the $(m-n)\times 2m$ matrix
$2\varGamma\varDelta$, where
\[
\varDelta\;=\;
\begin{pmatrix}
  x_1&y_1& \ldots & 0  &0 \\
  \vdots & \vdots & \ddots & \vdots & \vdots\\
  0  & 0 & \ldots &x_m &y_m
\end{pmatrix}.
\]
Assume that for the chosen point $\mb z\in\zp$ we have
$z_{j_1}=\ldots=z_{j_k}=0$, and the other complex coordinates are
nonzero. Then the rank of the gradient matrix
$2\varGamma\varDelta$ at $\mb z$ is equal to the rank of the
$(m-n)\times(m-k)$ matrix $\varGamma'$ obtained by deleting the
columns $j_1,\ldots,j_k$ from~$\varGamma$.

By Proposition~\ref{easyzp}~(a), the point $\mb z$ projects onto
$\mb x\in F_{j_1}\cap\ldots\cap F_{j_k}\ne\varnothing$. Let
$\iota\colon\R^{m-k}\rightarrow\R^m$ be the inclusion of the
coordinate subspace $\{\mb y\colon y_{j_1}=\ldots=y_{j_k}=0\}$,
and $\kappa\colon\R^m\rightarrow\R^k$ the projection onto the
quotient space. Then $\varGamma'=\varGamma\cdot\iota$, and
$A'=\kappa\cdot A_P$ is the $k\times n$ matrix formed by the row
vectors $\mb a_{j_1},\ldots,\mb a_{j_k}$. These vectors are
linearly independent by the assumption, and therefore $A'$ is of
rank~$k$ and $\kappa\cdot A_P$ is epic. We claim that
$\varGamma'=\varGamma\cdot\iota$ is also epic. Indeed, consider
the diagram
\[
\begin{CD}
  @.@.\begin{array}{c}0\\ \raisebox{2pt}{$\downarrow$}\\ \R^n\end{array}@.@.\\
  @.@.@VVA_PV@.@.\\
  0@>>>\R^{m-k}@>\iota>>\R^m@>\kappa>>\R^k @>>>0\\
  @.@.@VV \varGamma V@.@.\\
  @.@.\begin{array}{c}\R^{m-n}\\ \downarrow\\ 0\end{array}@.@.\\
\end{CD}
\]
Take $\alpha\in\R^{m-n}$. There is $\beta\in\R^m$ such that
$\varGamma\beta=\alpha$. Assume $\beta\notin\iota(\R^{m-k})$, then
$\gamma:=\kappa(\beta)\ne0$. Choose $\delta\in\R^n$ such that
$\kappa\cdot A_P(\delta)=\gamma$. Set $\eta:=A_P(\delta)\ne0$.
Hence, $\kappa(\eta)=\kappa(\beta)(=\gamma)$ and there is
$\xi\in\R^{m-k}$ such that $\iota(\xi)=\beta-\eta$. Then
$\varGamma\cdot\iota(\xi)=\varGamma(\beta-\eta)=\alpha-\varGamma\cdot
A_P\delta=\alpha$. Thus, $\varGamma\cdot\iota$ is epic, and
$\varGamma'$ has rank~$m-n$.
\end{proof}

We refer to $\zp$ as the \emph{moment-angle manifold corresponding
to}~$P$, or simply a \emph{polytopal moment-angle manifold}.

There is an indeterminacy in the choice of $\varGamma$. How to
make this choice more canonical?

\begin{construction}[1st method]
By reordering the facets and changing $\R^n$ by a linear
isomorphism we may achieve that the first $n$ facets of $P$ meet
at a vertex and $\mb a_1,\ldots,\mb a_n$ form the standard basis
of~$\R^n$. Then $A_P=\begin{pmatrix}I_n\\A_P^\star\end{pmatrix}$,
where $I_n$ is a unit $n\times n$ matrix, and $A_P^\star$ is an
$(m-n)\times n$ matrix. We may take
$\varGamma=\begin{pmatrix}-A_P^\star&I_{m-n}\end{pmatrix}$. This
choice of $\varGamma$ and the corresponding framing of $\zp$ was
used in~\cite{b-p-r07}.
\end{construction}

\begin{construction}[2nd method]\label{2ndme}
Assume for simplicity that all redundant inequalities
in~\eqref{zpqua} are of the form $1\ge0$ (that is, $\mb
a_i=\mathbf 0$ and $b_i=1$). The rows of the coefficient matrix
$\varGamma=(\gamma_{jk})$ of~\eqref{zpqua} form a basis in the
space of linear relations between the vectors $\mb a_1,\ldots,\mb
a_m$. Since these vectors are determined only up to a positive
multiple, we may assume that $|\mb a_i|=1$ if the $i$th inequality
is irredundant. Since $P$ is a convex polytope, one of the
relations between the $\mb a_i$'s is
$\sum_{i=1}^m(\mathop{\mathrm{vol}}F_i)\mb a_i=\mathbf 0$, where
$\mathop{\mathrm{vol}}F_i\ge0$ is the volume of the facet~$F_i$.
By rescaling the vectors $\mb a_i$ again we may achieve that
$\sum_{i=1}^m\mb a_i=\mathbf 0$ (this still leaves one scaling
parameter free), and choose the coefficients of this relation as
the last row of~$\varGamma$, that is, $\gamma_{m-n,k}=1$ for $1\le
k\le m$. Then the last of the quadrics in~\eqref{zpqua} takes the
form $\sum_{k=1}^m|z_k|^2=\sum_{k=1}^mb_k$. Summing up all $m$
inequalities of~\eqref{ptope}, using the fact that
$\sum_{i=1}^m\mb a_i=\mathbf 0$ and noting that at least one of
the inequalities is strict for some~$\mb x\in\R^n$, we obtain
$\sum_{k=1}^mb_k>0$. Using up the last free scaling parameter we
achieve that $\sum_{k=1}^mb_k=1$. Then the last quadric
in~\eqref{zpqua} becomes $\sum_{k=1}^m|z_k|^2=1$. Subtracting this
equation with an appropriate coefficient from the other equations
in~\eqref{zpqua} we finally transform~\eqref{zpqua} to
\begin{equation}\label{bmlink}
  \left\{\begin{array}{ll}
  \mb z\in\C^m\colon&\sum_{k=1}^m\gamma_{jk}|z_k|^2=0,\;\text{ for }1\le j\le m-n-1,\\[2mm]
  &\sum_{k=1}^m|z_k|^2=1.
  \end{array}\right\}
\end{equation}
We therefore have written $\zp$ as a transverse intersection of
$m-n-1$ \emph{homogeneous} quadrics with a unit sphere in~$\C^m$.
Such intersections were called \emph{links} in~\cite{bo-me06}. We
denote by $\varGamma^\star$ the submatrix of $\varGamma$ formed by
$\gamma_{jk}$ with $j\le m-n-1$, and denote the columns of
$\varGamma^\star$ by $\mb g_1,\ldots,\mb g_m$, so that
$\varGamma=\begin{pmatrix}\mb g_1&\ldots&\mb
g_m\\1&\ldots&1\end{pmatrix}$.

The identity $\varGamma A_P=0$ also means that the columns of
$A_P$ form a basis in the space of linear relations between the
columns of $\varGamma$, that is, a basis of solutions of the
homogeneous system
\begin{equation}\label{homsyst}
  \sum_{k=1}^m\mb g_ky_k=\textbf 0,\quad\sum_{k=1}^my_k=0.
\end{equation}
The passage from the vectors $\mb g_1,\ldots,\mb g_m\in\R^{m-n-1}$
to the vectors $\mb a_1,\ldots,\mb a_m\in\R^n$ forming the
transpose to a basis of solutions of~\eqref{homsyst} is known in
convex geometry as the \emph{Gale transform}.
\end{construction}

\begin{example}
Let $P$ be the standard $n$-simplex given by the equations
$x_i\ge0$ for $i=1,\ldots,n$ and $-x_1-\ldots-x_n+1\ge0$
in~$\R^n$. We therefore have $m=n+1$, $\mb a_i=\mb e_i$ (the $i$th
standard basis vector) for $i=1,\ldots,n$ and $\mb a_{n+1}=-\mb
e_1-\ldots-\mb e_n$, \
$A_P=\begin{pmatrix}I_n\\-{}1\:\ldots-\!1\end{pmatrix}$, \
$\varGamma=(1\ldots1)$, \ $\varGamma^\star=\varnothing$, and $\zp$
is a unit sphere in~$\C^{n+1}$.
\end{example}

The combinatorial structure of $P$ can be read from the
configuration of vectors $\mb g_1,\ldots,\mb g_m$ in $\R^{m-n-1}$
using the following lemma.

%
%

\begin{lemma}\label{propgd}
Let $P$ be a simple polytope~\eqref{ptope}. The following
conditions are equivalent:
\begin{itemize}
\item[(a)]
$F_{i_1}\cap\ldots\cap F_{i_k}\ne\varnothing$,

\item[(b)]
$\mathbf 0\in\conv\bigr(\mb g_j\colon
j\notin\{i_1,\ldots,i_k\}\bigl)$.
\end{itemize}
Here $\conv(\cdot)$ denotes the convex hull of a set of points.
\end{lemma}
\begin{proof}
Given $\mb x\in P$, set $\mb y=i_P(\mb x)$ and $y_i=\langle\mb
a_i,\mb x\rangle+b_i$ for $1\le i\le m$. Assume (a) is satisfied.
Choose $\mb x\in F_{i_1}\cap\ldots\cap F_{i_k}$;
then $y_j=0$ for $j\in\{i_1,\ldots,i_k\}$ and $y_j\ge0$ for
$j\notin\{i_1,\ldots,i_k\}$. Now $\varGamma^\star\mb y=\mathbf 0$
implies
\[
  \mb g_1y_1+\ldots+\mb
  g_my_m=\sum_{j\notin\{i_1,\ldots,i_k\}}\mb g_j y_j=\mathbf 0,
\]
which is equivalent to condition~(b). Conversely, if (b) is
satisfied, then there is $\mb y\in\R^m$ with $y_j=0$ for
$j\in\{i_1,\ldots,i_k\}$ and
$y_j\ge0$ for $j\notin\{i_1,\ldots,i_k\}$ such that $\varGamma\mb
y=\begin{pmatrix}\mathbf 0\\1\end{pmatrix}$. Therefore, $\mb
y=i_P(\mb x)$ for some $\mb x\in P$. For this $\mb x$ we have
$\langle\mb a_j,\mb x\rangle+b_j=0$ if $j\in\{i_1,\ldots,i_k\}$,
which implies that $\mb x\in F_{i_1}\cap\ldots\cap F_{i_k}$.
\end{proof}


The \emph{polar set} of a polyhedron given by~\eqref{ptope} is
defined as
\[
  P^*:=\bigl\{\mb u\in\R^n\colon\langle\mb u,\mb x\rangle\ge-1
  \quad\text{for all }\mb x\in P\bigr\}.
\]
If $P$ is a convex polytope and $\mathbf 0$ is in the interior of
$P$ (which implies that $b_i>0$ for all~$i$), then $P^*$ is also a
convex polytope given by
\[
  P^*=\conv\Bigl(\frac{\mb a_1}{b_1},\ldots,\frac{\mb a_m}{b_m}\Bigr),
\]
where $\conv(\cdot)$ denotes the convex hull of a set of points.
$P^*$ is called the \emph{polar} (or \emph{dual}) \emph{polytope}
of~$P$. If $P$ is simple, then $P^*$ is \emph{simplicial} (all
proper faces are simplices), and vice versa. In this case we have
that $F_{i_1}\cap\ldots\cap F_{i_k}$ is an $(n-k)$-face of~$P$ if
and only if $\conv\bigl(\frac{\mb
a_{i_1}}{b_{i_1}},\ldots,\frac{\mb a_{i_k}}{b_{i_k}}\bigr)$ is a
$k$-face of~$P^*$.

A set of vectors $\mb g_1,\ldots,\mb g_m\in\R^{m-n-1}$ satisfying
the conditions of Lemma~\ref{propgd} is called a \emph{Gale
diagram} of~$P^*$.

\section{From quadrics to polytopes}\label{fqtp}
Here we briefly review a construction of~\cite{bo-me06} which
associates a simple polytope to every link~\eqref{bmlink}.

For every set of vectors $\mb g_1,\ldots,\mb g_m\in\R^{m-n-1}$
there is the associated link:
\[
  \mathcal L:=\left\{\begin{array}{ll}
  \mb z\in\C^m\colon&\sum_{k=1}^m\mb g_k|z_k|^2=\textbf 0,\\[2mm]
  &\sum_{k=1}^m|z_k|^2=1.
  \end{array}\right\}
\]
We observe that if the set $\mb g_1,\ldots,\mb g_m$ is obtained
from a simple polytope~\eqref{ptope} as described in
Construction~\ref{2ndme}, and therefore $\mathcal L$ is an
embedding of the corresponding moment-angle manifold~$\zp$, then
the following two conditions are satisfied:
\begin{itemize}
\item[(i)] $\textbf 0\in\conv(\mb g_1,\ldots,\mb g_m)$;
\item[(ii)] if $\textbf 0\in\conv(\mb g_{i_1},\ldots,\mb g_{i_k})$ then $k\ge
m-n$.
\end{itemize}
(Indeed, (i) expresses the fact that $\mathcal L$ is nonempty,
and~(ii) easily follows from Lemma~\ref{propgd}.) The converse
statement also holds:

\begin{proposition}[{\cite[Lemma~0.12]{bo-me06}}]
Let $\mathcal L$ be a link satisfying conditions~{\rm(i)}
and~{\rm(ii)} above. Then there is a simple polytope~\eqref{ptope}
such that $\mathcal L$ is an embedding of the moment-angle
manifold~$\zp$. In particular the intersection of quadrics
defining~$\mathcal L$ is nondegenerate.
\end{proposition}
\begin{proof}
The torus $\T^m$ acts on $\mathcal L$ coordinatewise, and the
quotient $\mathcal L/\T^m$ is the set $Q$ of nonnegative solutions
of the system
\begin{equation}\label{qsyst}
  \sum_{k=1}^m\mb g_ky_k=\textbf 0,\quad\sum_{k=1}^my_k=1.
\end{equation}
Condition (ii) implies that this system has maximal
rank~$m-n$~\cite[Lemma~0.3]{bo-me06}, and therefore $Q$ is the
intersection of an $n$-dimensional affine plane with~$\R^m_\ge$.
Since $Q$ is bounded (as $\mathcal L$ is compact) and contains a
nonzero point $\mb y$ (by condition~(i)), it is a convex
$n$-dimensional polytope. Writing a general solution of
system~\eqref{qsyst} we obtain a presentation~\eqref{ptope} for
the polytope~$Q$, with $\mb a_1,\ldots,\mb a_m$ being a Gale
transform of $\mb g_1,\ldots,\mb g_m$ and $\sum_{k=1}^mb_k=1$.
This implies that $\mb g_1,\ldots,\mb g_m$ is a Gale diagram of
the dual polytope~$Q^*$. Now condition~(ii) and Lemma~\ref{propgd}
imply that at most $n$ facets of $Q$ may meet, hence $Q$ is
simple.
\end{proof}

\begin{example}\label{prodsimex}
Let $m-n-1=1$, so that $\mb g_i\in\R$. Condition~(ii) implies that
all $\mb g_i$ are nonzero; assume that there are $p$ positive and
$q=m-p$ negative numbers among them. Then condition~(i) implies
that $p>0$ and $q>0$. Therefore $\mathcal L$ is diffeomorphic to
\[
  \left\{\begin{array}{ll}
  \mb z\in\C^m\colon&|z_1|^2+\ldots+|z_p|^2-|z_{p+1}|^2-\ldots-|z_m|^2=0,\\[1mm]
  &|z_1|^2+\ldots+|z_m|^2=1.
  \end{array}\right\}
\]
The first equation specifies a cone over $S^{2p-1}\times
S^{2q-1}$, so that $\mathcal L\cong S^{2p-1}\times S^{2q-1}$. The
corresponding polytope is either a simplex $\Delta^{m-2}$ with one
redundant inequality (if $p=1$ or $q=1$) or a product
$\Delta^{p-1}\times\Delta^{q-1}$.
\end{example}

\newpage

\part{Other constructions of the moment-angle manifold}
In the previous lecture we defined the moment-angle manifold $\zp$
corresponding to a simple polytope~$P$, and embedded it into
$\C^m$ as a nondegenerate intersection of $m-n$ real quadrics.
There are several other constructions of the moment-angle
manifold, some of which admit interesting generalisations and lead
to new applications.

\section{Identification space} Denote by
$[m]$ the $m$-element set $\{1,\ldots,m\}$. For any subset
$I\subset[m]$ define the corresponding \emph{coordinate subgroup}
$T^I$ in the torus $\T^m$ as
\begin{equation}\label{tisub}
  T^I:=\{\mb t=(t_1,\ldots,t_m)\in\T^m\colon t_j=1\text{ for
  }j\notin I\}.
\end{equation}
In particular, $T^\varnothing$ is the trivial subgroup~$\{1\}$.

We consider the map $\R_\ge\times\T\to\C$ defined by $(y,t)\mapsto
yt$. Taking product we obtain a map $\R^m_\ge\times\T^m\to\C^m$.
The preimage of a point $\mb z\in\C^m$ under this map is $\mb
y\times T^{I(\mb z)}$, where $y_i=|z_i|$ for $1\le i\le m$ and
$I(\mb z)\subset[m]$ is the set of zero coordinates of~$\mb z$.
Therefore, $\C^m$ can be identified with the quotient space
\[
  \R^m_\ge\times\T^m/{\sim\:}\quad\text{where }(\mb y,\mb t_1)\sim(\mb y,\mb
  t_2)\text{ if }\mb t_1^{-1}\mb t_2\in T^{I(\mb y)}.
\]

The space $\zp$ was originally defined in~\cite{da-ja91} as a
similar identification space. Given $p\in P$, set
$I_p=\{i\in[m]\colon p\in F_i\}$ (the set of facets
containing~$p$).

\begin{proposition}
The moment-angle manifold $\zp$ is $\T^m$-equivariantly
homeomorphic to the quotient
\[
  P\times\T^m/{\sim\:}\quad\text{where }(p,\mb t_1)\sim(p,\mb t_2)\:\text{ if }\:\mb t_1^{-1}\mb t_2\in
  T^{I_p}.
\]
\end{proposition}
\begin{proof}
It follows from~\eqref{cdiz} that $\zp$ is $\T^m$-homeomophic to
$i_P(P)\times\T^m/{\sim\:}$, and a point $p\in P$ is mapped by
$i_P$ to $\mb y\in\R^m_\ge$ with $I_p=I(\mb y)$.
\end{proof}

\begin{corollary}
The $\T^m$-equivariant topological type of the manifold $\zp$
depends only on the combinatorial type (the face poset) of the
polytope~$P$.
\end{corollary}

\section{Moment-angle complex} We consider the unit polydisc
in $\C^m$:
\[
  \D^m=\bigl\{ (z_1,\ldots,z_m)\in\C^m\colon |z_i|\le1,\quad i=1,\ldots,m
  \bigr\}.
\]
The quotient $\D^m/\T^m$ is the standard \emph{unit $m$-cube}
$\I^m=[0,1]^m$, and we have an identification
$\I^m\times\T^m/{\sim\:}\cong\D^m$ like that considered in the
previous subsection.

\hangindent=28mm \hangafter=2 Let $P$ be an $n$-dimensional
combinatorial simple polytope, and let $V(P)$ be its set of
vertices (it is helpful to have a geometric
presentation~\eqref{ptope} in mind, although only the
combinatorial structure of $P$ is relevant here). Every such $P$
may be represented as a union $\bigcup C_v$ of combinatorial
$n$-cubes, with one such cube $C_v$ for vertex $v\in V(P)$. See
the figure on the left where a pentagon is split into 5
quadrilaterals (combinatorial 2-cubes); the details of this
self-evident construction can be found in~\cite[\S4.2]{bu-pa02}.

\parindent=0.03\textwidth\raisebox{1.5\baselineskip}[0pt]
{%
\begin{picture}(20,20)
  \put(4.5,1.5){\line(-1,3){4.5}}
  \put(0,15){\line(2,1){10}}
  \put(10,20){\line(2,-1){10}}
  \put(20,15){\line(-1,-3){4.5}}
  \put(15.5,1.5){\line(-1,0){11}}
  \put(10,10){\line(-4,-1){7.75}}
  \put(10,10){\line(-2,3){5}}
  \put(10,10){\line(2,3){5}}
  \put(10,10){\line(4,-1){7.75}}
  \put(10,10){\line(0,-1){8.5}}
  \put(2.5,0){\small$v$}
  \put(5.5,5){\small$C_v$}
\end{picture}%
}

\vspace{-\baselineskip}

Every vertex $v\in P$ can be written as an intersection of $n$
facets: $v=F_{i_1}\cap\ldots\cap F_{i_n}$, and we denote
$I_v:=\{i_1,\ldots,i_n\}$. Consider the subset
$C_v\times\T^m/{\sim\:}\subset P\times\T^m/{\sim\:}\cong\zp$. We
have
\[
  C_v\times\T^m/{\sim\:}=
  \bigl(C_v\times T^{I_v}/{\sim\:}\bigr)\times T^{[m]\setminus I_v}
  \cong\D^n\times\T^{m-n}.
\]
We therefore may identify $\zp$ with the union
\begin{equation}\label{zpbv}
  \bigcup_{v\in V(P)} B_v\subset\D^m,
\end{equation}
where
\[
  B_v:=\bigl\{(z_1,\ldots,z_m)\in
  \D^m\colon |z_j|=1\text{ if }v\notin
  F_j\bigl\}\cong\D^n\times\T^{m-n}.
\]

Now let $\mathcal K_P$ be the boundary $\partial P^*$ of the dual
simplicial polytope. It can be viewed as a simplicial complex on
the set $[m]$, whose simplices are subsets $I=\{i_1,\ldots,i_k\}$
such that $F_{i_1}\cap\ldots\cap F_{i_k}\ne\varnothing$ in~$P$.
(Note that a redundant inequality in~\eqref{ptope} whose
corresponding facet $F_i$ is empty gives rise to a \emph{ghost
vertex} of $\mathcal K_P$, i.e. a one-element subset $\{i\}$ of
$[m]$ which is not a vertex of~$\mathcal K_P$.) Then we may
rewrite~\eqref{zpbv} as $\zp\cong\bigcup_{I\in\sK_P}B_I$, where
\[
  B_I:=\bigl\{(z_1,\ldots,z_m)\in
  \D^m\colon |z_j|=1\text{ for }j\notin I\bigl\}.
\]
This construction admits the following generalisation.

\begin{definition}[\cite{bu-pa02}]
Let $\sK$ be a simplicial complex on the set~$[m]$. The
corresponding \emph{moment-angle complex} $\zk$ is defined as
\[
  \zk:=\bigcup_{I\in\sK}B_I\subset\D^m.
\]
\end{definition}
Note that $\dim\zk=m+\dim\sK+1$.

\begin{example}\label{zkexam}
1. Let $\sK=\partial\Delta^n$, the boundary of an $n$-simplex
(which is dual to $P=\Delta^n$) and $m=n+1$. Then
\[
  \zk=(\D{\times}\ldots{\times}\D{\times}\T)\cup(\D{\times}\ldots{\times}
  \T{\times}\D)\cup\ldots\cup(\T{\times}\D{\times}\ldots{\times}\D)=
  \partial\D^{n+1}\cong S^{2n+1}.
\]

2. Let $\sK$ be a disjoint union of $m$ points (this example is
not of the form $\sK_P$ if $m\ne2$). Then $\zk$ is given as the
following $m+1$-dimensional subspace in~$\D^m$:
\[
  \zk=(\D{\times}\T{\times}\ldots{\times}\T)\cup(\T{\times}\D{\times}\ldots{\times}
  \T)\cup\ldots\cup(\T{\times}\T{\times}\ldots{\times}\D).
\]
\end{example}

\begin{lemma}[{\cite[Lemma~6.13]{bu-pa02}}]
If $\sK$ is a triangulation (simplicial subdivision) of an
$(n-1)$-dimensional sphere, then $\zk$ is a closed topological
manifold of dimension~$m+n$.
\end{lemma}

\begin{remark}
If $\sK=\sK_P$ for a simple polytope~$P$, then $\zk\cong\zp$ and
therefore the manifold $\zk$ can be smoothed by
Theorem~\ref{zpsmooth}. However, there are many sphere
triangulations $\sK$ which are not of the form $\sK_P$ for
any~$P$. The question of whether the corresponding $\zk$ can be
smoothed is open in general (see~\cite{pa-us10} for a construction
of smooth structures on some nonpolytopal~$\zk$).
\end{remark}

\section{Coordinate subspace arrangement complement}\label{csac} A
\emph{coordinate subspace} in $\C^m$ is determined by a subset
$I=\{i_1,\ldots,i_m\}\subset[m]$:
\[
  L_I:=\{\mb z\in\C^m\colon z_{i_1}=\ldots=z_{i_k}=0\}.
\]
Every simplicial complex $\sK$ on $[m]$ defines an
\emph{arrangement} of coordinate subspaces in~$\C^m$ and its
\emph{complement}
\begin{equation}\label{uset}
  U(\sK):=\C^m\Big\backslash\bigcup_{I\notin\sK}L_I.
\end{equation}

\begin{example}\label{complexam}
1. If $\sK=\partial\Delta^{m-1}$ then $U(\sK)=\C^m\setminus\{\mb
z\colon z_1=\ldots=z_m=0\}=\C^m\setminus\{\mathbf0\}$.

2. If $\sK$ is a disjoint union of $m$ points, then
\[
  U(\sK)=\C^m\Big\backslash\bigcup_{1\le i<j\le m}
  \bigl\{\mb z\colon z_i=z_j=0\bigr\}
\]
is the complement to all coordinate planes of codimension~2.
\end{example}

\begin{proposition}
Every complement to a set of coordinate subspaces in $\C^m$ has
the form $U(\sK)$ for some~$\sK$.
\end{proposition}
\begin{proof}
Let $U\subset\C^m$ be such a complement. Then we have $U=U(\sK)$
where $\sK=\{I\subset[m]\colon L_I\cap U\ne\varnothing\}$.
\end{proof}

Observe that $\zk\subset U(\sK)$ for every~$\sK$.

\begin{theorem}[{\cite[Th.~8.9]{bu-pa02}}]
There is a $\T^m$-equivariant deformation retraction
$U(\sK)\to\zk$.
\end{theorem}

\begin{example}
Let $\sK=\partial\Delta^{m-1}$. Then
$U(\sK)=\C^m\setminus\{\mathbf 0\}$ retracts onto $\zk=S^{2m-1}$.
\end{example}

\section{Level set for a moment map}\label{levelset} The moment-angle
manifold $\zp$ is closely related to the construction of
\emph{Hamiltonian toric manifolds} via symplectic reduction.

Recall that a \emph{symplectic manifold} $(W,\omega)$ is a smooth
(but not necessaririly compact) manifold $W$ with a closed 2-form
$\omega$ which is nondegenerate at every point. Assume that a
torus $T$ acts on $W$ preserving the symplectic form~$\omega$.
Denote by $\mathfrak t$ the Lie algebra of~$T$ (this algebra is
commutative and therefore trivial, but the construction may be
extended to noncommutative Lie group actions). For any
$u\in\mathfrak t$ denote by $\xi_u$ the corresponding
$T$-invariant vector field on~$W$. The $T$-action is
\emph{Hamiltonian} if the 1-form $\omega(\:\cdot\:,\xi_u)$ is
exact for every $u\in\mathfrak t$. In other words, there is a
function $H_u$ on $W$, called a \emph{Hamiltonian}, such that
$\omega(\xi,\xi_u)=dH_u(\xi)$ for every vector field $\xi$ on~$W$.
The \emph{moment map}
\[
  \mu\colon W\to\mathfrak t^*,\qquad (x,u)\mapsto
  H_{u}(x)
\]
is therefore defined. Its image $\mu(W)$ is a convex polyhedron (a
convex polytope if $W$ is compact) by a theorem of Atiyah and
Guillemin--Sternberg~\cite{guil94}.

\begin{example}\label{simcm}
A basic example is $W=\C^m$ with the symplectic form
$\omega=2\sum_{k=1}^mdx_k\wedge dy_k$ where $z_k=x_k+iy_k$. The
coordinatewise action of $\T^m$ is Hamiltonian and the moment map
$\mu\colon\C^m\to\R^m$ is given by
$\mu(z_1,\ldots,z_m)=(|z_1|^2,\ldots,|z_m|^2)$.
\end{example}

A simple polytope~\eqref{ptope} is called \emph{Delzant} if the
vectors $\mb a_i$ have integral coordinates and for every vertex
$v=F_{i_1}\cap\ldots\cap F_{i_n}$ the set $\{\mb
a_{i_1},\ldots,\mb a_{i_n}\}$ is a basis of the integral lattice
$\Z^n\subset\R^n$.

Assume now that $P$ is a Delzant polytope. We also assume for
simplicity that there are no redundant inequalities
in~\eqref{ptope} (redundant inequalities may be also taken into
account by simple modifications to the constructions below). Let
$\Lambda\colon\R^m\to\R^n$, \ $\mb e_i\mapsto\mb a_i$, be the
traspose of~$A_P$. Since $P$ is Delzant, it restricts to a map of
integral lattices $\Z^m\to\Z^n$ and defines a map of tori
$\T^m\to\T^n$, which we continue denoting~$\Lambda$. Consider
$K:=\Ker(\Lambda\colon\T^m\to\T^n)$. Because of the Delzant
condition, $K$ is isomorphic to an $(m-n)$-torus.

\begin{example}
By restricting the $\T^m$-action of Example~\ref{simcm} to $K$ we
obtain a Hamiltonian action whose moment map is given by the
composition
\[
  \mu_K\colon\C^m\longrightarrow\R^m\stackrel\varGamma\longrightarrow\R^{m-n},
\]
where $\varGamma=(\gamma_{jk})$ is defined by~\eqref{iprn}. The
quadratic forms $\sum_{k=1}^m\gamma_{jk}|z_k|^2$ for $1\le j\le
m-n$ constitute a basis in the space of Hamiltonian functions,
and~\eqref{zpqua} implies the following

\begin{proposition}
If $P$ is a Delzant polytope then the moment-angle manifold $\zp$
is identified with the level set $\mu_K^{-1}(\mb c)$ of the moment
map $\mu_K$ for the Hamiltonian action of $K$ on~$\C^m$, where
$\mb c=(c_1,\ldots,c_{m-n})$ and $c_j=\sum_{k=1}^m\gamma_{jk}b_k$.
\end{proposition}

Note that $\mb c$ is a \emph{regular value} of the moment map
$\mu_K$ by Theorem~\ref{zpsmooth}.
\end{example}

\begin{lemma}\label{freeact}
If $P$ is Delzant then the action of $K\subset\T^m$ on $\zp$ is
free.
\end{lemma}
\begin{proof}
A point $\mb z\in\C^m$ has a nontrivial isotropy subgroup with
respect to the $\T^m$-action only if some of the coordinates of
$\mb z$ vanish. These $\T^m$-isotropy subgroups are of the form
$T^{I(\mb z)}$, see~\eqref{tisub}, where $I(\mb z)$ is the set of
zero coordinates of~$\mb z$. If $\mb z\in i_Z(\zp)$ then
$\bigcap_{i\in I(\mb z)}F_i\ne\varnothing$, and the restriction of
$\Lambda\colon\T^m\to\T^n$ to every such $T^{I(\mb z)}$ is an
injection by the Delzant condition. Therefore, $K=\Ker \Lambda$
intersects every $\T^m$-isotropy subgroup only at the unit.
\end{proof}

\begin{construction}[Symplectic reduction]
The manifold $\mu_K^{-1}(\mb c)\cong\zp$ fails to be symplectic as
the restriction of $\omega$ to $\mu_K^{-1}(\mb c)$ is degenerate.
However it may be shown~\cite{guil94} that the quotient
$\mu_K^{-1}(\mb c)/K$ supports a nondegenerate 2-form $\omega'$
satisfying the condition $p^*\omega'=i_Z^*\omega$, where $p\colon
\mu_K^{-1}(\mb c)\to\mu_K^{-1}(\mb c)/K$ is the projection.
Therefore $\bigl(\mu_K^{-1}(\mb c)/K,\omega'\bigr)$ is a
symplectic manifold of dimension $2n$. It has a residual
Hamiltonian action of the $n$-torus~$\T^m/K$. The manifold
$M_P:=\mu_K^{-1}(\mb c)/K$ is referred to as a \emph{Hamiltonian
toric manifold}. The passage from $(\C^m,\omega)$ to
$(M_P,\omega')$ is known as the \emph{symplectic reduction} of
$\C^m$ by the action of $K$.
\end{construction}

Hamiltonian toric manifolds are closely related to nonsingular
projective toric varieties in algebraic geometry.

A \emph{toric variety} is a normal algebraic variety $X$ on which
an algebraic torus $(\C^\times)^n$ acts with a dense orbit,
see~\cite{dani78}.

A set of vectors $\mb a_1,\ldots,\mb a_k\in\R^n$ defines a convex
polyhedral \emph{cone}
\[
  \sigma=\{\mu_1\mb a_1+\ldots+\mu_k\mb
  a_k\colon\mu_i\in\R_\ge\}.
\]
A cone is \emph{rational} if its generating vectors can be chosen
from the integral lattice $\Z^n\subset\R^n$, and is \emph{strongly
convex} if it does not contain a line. A cone is \emph{simplicial}
(respectively, \emph{regular}) if it is generated by a part of
basis of $\R^n$ (respectively, $\Z^n$).

A \emph{fan} is a finite collection
$\Sigma=\{\sigma_1,\ldots,\sigma_s\}$ of strongly convex cones in
some $\R^n$ such that every face of a cone in $\Sigma$ belongs to
$\Sigma$ and the intersection of any two cones in $\Sigma$ is a
face of each. A fan $\Sigma$ is \emph{rational} (respectively,
\emph{simplicial}, \emph{regular}) if every cone in $\Sigma$ is
rational (respectively, simplicial, regular). A fan
$\Sigma=\{\sigma_1,\ldots,\sigma_s\}$ is called \emph{complete} if
$\sigma_1\cup\ldots\cup\sigma_s=\R^n$.

\begin{example}
Let \eqref{ptope} be a simple polytope. The \emph{normal fan}
$\Sigma_P$ of $P$ consists of cones spanned by those sets of
vectors $\mb a_{j_1},\ldots,\mb a_{j_k}$ for which intersection
$F_{j_1}\cap\ldots\cap F_{j_k}$ is nonempty. It is a complete
simplicial fan in~$\R^n$. If $P$ is Delzant, then $\Sigma_P$ is
rational and regular.
\end{example}

As is well known in algebraic geometry, toric varieties are
classified by rational fans~\cite{dani78}. Under this
correspondence, complete fans give rise to compact varieties,
normal fans of polytopes to projective varieties, regular fans to
nonsingular varieties, and simplicial fans to varieties with mild
(orbifold-type) singularities.

There is the following algebraic version of symplectic reduction,
which is now commonly referred to as the `Cox construction',
although it takes origin in the work of several
authors~\cite{cox95}.

\begin{construction}
Assume that $\Sigma$ is a complete rational simplicial fan
in~$\R^n$ with $m$ one-dimensional cones generated by primitive
vectors $\mb a_1,\ldots,\mb a_m\in\Z^m$. The \emph{underlying
simplicial complex} of $\Sigma$ is defined as
\[
  \sK_\Sigma:=\bigl\{I=\{i_1,\ldots,i_k\}\subset[m]\colon
  \mb a_{i_1},\ldots,\mb a_{i_k}\text{ span a cone of }
  \Sigma\bigr\}.
\]
It may be viewed geometrically as the intersection of $\Sigma$
with a unit sphere.

Let $\Lambda_\C\colon(\C^\times)^m\to(\C^\times)^n$ be the map of
algebraic tori corresponding to the map $\Z^m\to\Z^n$, \ $\mb
e_i\mapsto\mb a_i$. Set $G:=\Ker\Lambda_\C$. This is an
$(m-n)$-dimensional algebraic subgroup in $(\C^\times)^m$, hence,
it is isomorphic to a product of $(\C^\times)^{m-n}$ and a finite
group (the finite group is trivial if the fan is regular). The
group $G$ acts almost freely (with finite isotropy subgroups) on
the open set $U(\sK_\Sigma)$ of~\eqref{uset}; moreover, this
action is free if $\Sigma$ is a regular fan. This is proved in the
same way as in Lemma~\ref{freeact}.

The toric variety associated to the fan $\Sigma$ is defined as the
quotient $X_\Sigma:=U(\sK_\Sigma)/G$. It is a complex algebraic
variety of dimension~$n$. The variety $X_\Sigma$ is nonsingular
whenever $\Sigma$ is regular; otherwise it has only orbifold-type
singularities (locally isomorphic to a quotient of $\C^n$ by a
finite group). The quotient algebraic torus
$(\C^\times)^m/G\cong(\C^\times)^n$ acts on $X_\Sigma$ with a
dense orbit.

The variety $X_\Sigma$ is projective if and only if $\Sigma$ is
the normal fan of a polytope~$P$; in this case we shall denote the
variety by~$X_P$.

The Cox construction extends to noncomplete and nonsimplicial fans
(in the latter case the ordinary quotient needs to be replaced by
the \emph{categorical} one), but we shall not need this generality
here.
\end{construction}

Now if $P$ is a Delzant polytope, then the nonsingular projective
toric variety $X_P$ is symplectic and is $\T^n$-equivariantly
symplectomorphic (for an appropriate choice of the symplectic
form) to the Hamiltonian toric manifold~$M_P$. In other words, the
quotient of the open set $U(\sK_{\Sigma_P})\subset\C^m$ by the
action of a noncompact group $G$ can be identified with the
quotient of the compact subset $i_Z(\zp)\subset U(\sK_{\Sigma_P})$
by a compact subgroup $K\subset G$~\cite[App.~1]{guil94}.

\newpage

\part{Topology of moment-angle complexes}
The topology of moment-angle manifolds $\zp$ and complexes $\zk$
is quite complicated even for relatively small and easily
described polytopes $P$ and complexes~$\sK$. We give evidences to
this by describing the cohomology ring of $\zk$ and then providing
explicit homotopy and diffeomorphism types for some series of $P$
and~$\sK$.

\section{The cohomology ring}
We continue denoting by $\sK$ a simplicial complex on~$[m]$. We
denote by $\Z[v_1,\ldots,v_m]$ the polynomial ring and by
$\Lambda[u_1,\ldots,u_m]$ the exterior ring with integer
coefficients. Given a subset $I=\{i_1,\ldots,i_k\}\subset[m]$ we
denote by $v^I$ the square-free monomial $v_{i_1}\cdots v_{i_k}$.
We use `dg ring' as an abbreviation for `differential graded ring'
and similarly for abelian groups ($\Z$-modules).

\begin{definition}\label{deffr}
The \emph{face ring} (also known as the \emph{Stanley--Reisner
ring}) of~$\sK$ is the following quotient of the polynomial ring
on $m$ generators:
\[
  \Z[\mathcal K]=\Z[v_1,\ldots,v_m]/(v^I\colon I\notin\sK).
\]
We make $\Z[\mathcal K]$ a graded ring by setting
$\deg v_i=2$ for all~$i$.
\end{definition}

\begin{example}
1. If $\sK=\partial\Delta^{m-1}$ then
$\Z[\sK]=\Z[v_1,\ldots,v_m]/(v_1\cdots v_m)$.

2. If $\sK$ is $m$ points, then
$\Z[\sK]=\Z[v_1,\ldots,v_m]/(v_iv_j\text{ for }1\le i<j\le m)$.
\end{example}

We abbreviate $\Z[v_1,\ldots,v_m]$ to $\Z[m]$ to make formulae
shorter. The face ring $\Z[\sK]$ is a $\Z[m]$-module via the
quotient projection. Its \emph{free resolution} is an exact
sequence of finitely generated $\Z[m]$-modules
\[
  0\longrightarrow R^{-m}\longrightarrow\ldots\longrightarrow
  R^{-1}\longrightarrow R^0\longrightarrow\Z[\sK]\longrightarrow0
\]
in which all $R^{-i}$ are free modules. The \emph{$(-i)$th $\Tor$
group} $\Tor^{-i}_{\Z[m]}(\Z[\sK],\Z)$ is defined as the $(-i)$th
cohomology group of the complex
\[
  0\longrightarrow R^{-m}\otimes_{\Z[m]}\Z\longrightarrow\ldots\longrightarrow
  R^{-1}\otimes_{\Z[m]}\Z\longrightarrow
  R^0\otimes_{\Z[m]}\Z\longrightarrow0.
\]

The groups $\Tor^{-i}_{\Z[m]}(\Z[\sK],\Z)$ acquire an internal
grading from the grading in $\Z[m]$ and $\Z[\sK]$. We define
\[
  \Tor_{\Z[m]}(\Z[\sK],\Z):=\bigoplus_{i=0}^m\Tor^{-i}_{\Z[m]}(\Z[\sK],\Z),
\]
which therefore has two gradings; the \emph{total degree} is the
sum of these two gradings. Morever, $\Tor_{\Z[m]}(\Z[\sK],\Z)$ has
a canonical multiplication turning it into a graded ring with
respect to the total degree (see~\cite{bu-pa02}).

\begin{theorem}\label{zkcoh}
The cohomology ring of the moment-angle complex $\zk$ is given by
the isomorphisms
\[
\begin{aligned}
  H^*(\zk;\Z)&\cong\Tor_{\Z[v_1,\ldots,v_m]}(\Z[\sK],\Z)\\
  &\cong H\bigl[\Lambda[u_1,\ldots,u_m]\otimes\Z[\sK],d\bigr],
\end{aligned}
\]
where the latter ring is the cohomology of the dg ring whose
grading and differential are given by
\[
  \deg u_i=1,\;\deg v_i=2;\quad du_i=v_i,\;dv_i=0.
\]
\end{theorem}

This theorem was proved in~\cite{bu-pa99} (with coefficients in a
field), see also~\cite{bu-pa02}. We give a sketch of the proof for
$\Z$ coefficients, following~\cite{pano08}. This proof first
appeared in~\cite{b-b-p04}. Another proof of the integral version
appeared in~\cite{fran06}.

\begin{proof}[Sketch of proof of Theorem~\ref{zkcoh}]
We only prove that $H^*(\zk;\Z)$ is isomorphic to
$H\bigl[\Lambda[u_1,\ldots,u_m]\otimes\Z[\sK],d\bigr]$; the fact
that the latter ring is isomorphic to the $\Tor$ is a standard
application of the Koszul resolution. The proof is split into 4
steps.

\smallskip

{\noindent \hangindent=27mm \hangafter=4 \emph{Step 1: cellular
decomposition of $\zk$.} We consider the following decomposition
of the disc $\D$ into 3 cells: the point $1\in\D$ is a 0-cell; the
complement to $1$ in the boundary circle is a 1-cell, which we
denote~$T$; and the interior of $\D$ is a 2-cell, which we
denote~$D$. By taking product we obtain a cellular decomposition
of~$\D^m$ whose cells are parametrised by pairs of subsets
$I,J\subset [m]$ with $I\cap J=\varnothing$: the set $I$
parametrises the $T$-cells in the product and $J$ parametrises the
$D$-cells. We denote the cell of $\D^m$ corresponding to a pair
$I,J$ by $\varkappa(I,J)$; it is a product of $|I|$ cells of $T$
type and $|J|$ cells of $D$ type. Then $\zk$ includes as a
cellular subcomplex in $\D^m$; we have $\varkappa(I,J)\subset\zk$
whenever $J\in\sK$.

}
\noindent\raisebox{\baselineskip}[0pt]
{%
\begin{picture}(20,20)
  \put(12,10){\circle{20}}
  \put(19,10){\circle*{1.3}}
  \put(20,9){$1$}
  \put(11,9){$D$}
  \put(16,16.5){$T$}
\end{picture}%
}

\vspace{-0.9\baselineskip}

We denote by $C^*(\zk)$ the cellular cochain group of~$\zk$. It
has a basis of cochains $\varkappa(I,J)^*$ dual to the
corresponding cells.

\smallskip

\noindent\emph{Step 2: dg ring model for $C^*(\zk)$.} We consider
the following quotient dg ring:
\[
  R^*(\sK):=\Lambda[u_1,\ldots,u_m]\otimes\Z[\sK]\big/
  (u_iv_i,\;v_i^2,\quad\text{for }1\le i\le m).
\]
It has a finite rank as an abelian group, unlike
$\Lambda[u_1,\ldots,u_m]\otimes\Z[\sK]$. Namely, the monomials
$u^Iv^J$ with $I\cap J=\varnothing$ and $J\in\sK$ constitute a
basis of~$R^*(\sK)$. Define the map
\[
  g\colon R^*(\sK)\longrightarrow C^*(\zk),\quad
  u^Iv^J\mapsto\varkappa(I,J)^*.
\]
It is an isomorphism of dg groups by inspection. Therefore we have
an additive isomorphism $H[R^*(\sK)]\cong H^*(\zk)$.

\smallskip

\noindent\emph{Step 3:
$H\bigl[\Lambda[u_1,\ldots,u_m]\otimes\Z[\sK],d\bigr]\cong
H\bigl[R^*(\sK),d\bigr]$, i.e. $(u_iv_i,\;v_i^2,\;1\le i\le m)$ is
an acyclic ideal.} We have a pair of maps of dg groups:\\[-0.5\baselineskip]
\[
  \Lambda[u_1,\ldots,u_m]\otimes\Z[\sK]
  {\raisebox{-0.3\baselineskip}[0pt]{$\stackrel\varrho\longrightarrow$}\atop
  \raisebox{0.5\baselineskip}[0pt]{$\mathop{\longleftarrow}\limits_{\iota}$}} R^*(\sK)
\]
where $\varrho$ is the quotient projection (a ring map) and
$\iota$ sends $u^Iv^J$ to itself (it is a monomorphism of dg
groups, but not a ring map). We have $\varrho\cdot\iota=\id$ and
there is an explicitly defined map $s$ satisfying the identity
$ds+sd=\id-\iota\cdot\varrho$ (a cochain homotopy between $\id$
and $\iota\cdot\varrho$). It follows that $\varrho$ induces an
isomorphism in cohomology.

\smallskip

\noindent\emph{Step 4: $g\colon R^*(\sK)\to C^*(\zk)$ is a ring
isomorphism.} We already know from Step~2 that $g$ is an
isomorphism of dg groups. A ring structure in $C^*(\zk)$ is
defined by a choice of a cellular approximation for the diagonal
map $\Delta\colon\zk\to\zk\times\zk$. Consider the map
$\widetilde{\Delta}\colon\D\to\D\times\D$, defined in polar
coordinates $z=\rho e^{i\varphi}\in\D$, $0\le\rho\le1$,
$0\le\varphi<2\pi$ as follows:
\[
  \rho e^{i\varphi}\mapsto\left\{
  \begin{array}{ll}
    (1+\rho(e^{2i\varphi}-1),1)&\text{ for }0\le\varphi\le\pi,\\
    (1,1+\rho(e^{2i\varphi}-1))&\text{ for }\pi\le\varphi<2\pi.
  \end{array}
  \right.
\]
This is a cellular map (with respect to the cellular decompostion
of Step~1) homotopic to the diagonal
$\Delta\colon\D\to\D\times\D$. Taking an $m$-fold product, we
obtain a cellular approximation
$\widetilde{\Delta}\colon\D^m\to\D^m\times\D^m$ which restricts to
a cellular approximation for the diagonal map of $\zk$ for
arbitrary $\sK$. The ring structure in $C^*(\zk)$ defined by the
composition
\[
\begin{CD}
  C^*(\zk)\otimes C^*(\zk) @>\times>> C^*(\zk\times \zk)
  @>\widetilde{\Delta}^*>> C^*(\zk)
\end{CD}
\]
and induces the cup product in the cohomology of~$\zk$.

We therefore need to check that $g\colon R^*(\sK)\to C^*(\zk)$ is
a multiplicative map with respect to the ring structures in
$R^*(\sK)$ and $C^*(\zk)$. To do this we note that both ring
structures are functorial with respect to inclusions of simplicial
complexes, and $R^*(\Delta^{m-1})\to C^*(\D^m)$ is a ring
isomorphism by inspection (both rings are isomorphic to
$\Lambda[u_1,\ldots,u_m]\otimes\Z[m]/(u_iv_i,\;v_i^2,\;1\le i\le
m)$). The multiplicativity of $g$ for arbitrary $\sK$ follows by
considering the commutative diagram
\[
\begin{CD}
  R^*(\Delta^{m-1}) @>\text{\parbox{18pt}{ring\\[-.5\baselineskip]\hphantom{a}iso}}>> C^*(\D^m)\\
  @V\text{\parbox{18pt}{ring\\[-.5\baselineskip]\hphantom{a}epi}}VV
  @VV\text{\parbox{18pt}{ring\\[-.5\baselineskip]\hphantom{a}epi}}V\\
  R^*(\sK) @>g>\text{\parbox{15pt}{add.\\[-.5\baselineskip]\hphantom{i}iso}}> C^*(\zk).
\end{CD}
\]\\[-1.5\baselineskip]
\end{proof}

The bigrading in the Tor defines a bigrading in $H^*(\zk)$, and we
may define
\[
  H^{-i,2j}(\zk):=\Tor^{-i,2j}_{\Z[v_1,\ldots,v_m]}(\Z[\sK],\Z).
\]
This bigrading may be also induced from the dg ring
$\Lambda[u_1,\ldots,u_m]\otimes\Z[\sK]$ by setting $\bideg
u_i=(-1,2)$ and $\bideg v_i=(0,2)$. Moreover, this bigrading may
be refined to a $\Z\oplus\Z^m$-multigrading by setting
\[
  \mdeg u_i=(-1,2\mb e_i),\quad \mdeg v_i=(0,2\mb e_i)
\]
where $\mb e_i\in\Z^m$ is the $i$th basis vector. We denote by
$H^{-i,2\mb a}(\zk)$ the component of multidegree $(-i,2\mb a)$
for $\mb a\in\Z^m$. Since $H^*(\zk)\cong H[R^*(\sK)]$ (see the
proof of Theorem~\ref{zkcoh}), $H^{-i,2\mb a}(\zk)$ is nonzero
only if $\mb a\in\{0,1\}^m$, and such vectors $\mb a$ may be
identified with subsets $J\subset[m]$ by considering the unit
coordinates of $\mb a$. We therefore have
\[
  H^k(\zk)=\bigoplus_{-i+2j=k}H^{-i,2j}(\zk)=\bigoplus_{-i+2|J|=k}H^{-i,2J}(\zk)\quad
  \text{for } i,j,k\ge0,\;J\subset[m].
\]

Given $J\subset[m]$ denote by $\sK_J$ the corresponding \emph{full
subcomplex} of $\sK$ (the restriction of $\sK$ to~$J$).

\begin{corollary}
We have
\[
  H^k(\zk)\cong\bigoplus_{J\subset[m]}
  \widetilde H^{k-|J|-1}(\sK_J)
  \quad\text{and}\quad
  H^{-i,2J}(\zk)\cong\widetilde
  H^{|J|-i-1}(\sK_J),
\]
where $\widetilde H^p(\sK_J)$ denotes the $p$th reduced simplicial
cohomology group of~$\sK_J$.
\end{corollary}
\begin{proof}
The second formula follows from the fact that the differential $d$
preserves the $\Z^m$ part of the $\Z\oplus\Z^m$-multigrading in
$R^*(\sK)$, and the cohomology of $R^{*,2J}(\sK)$ is isomorphic to
$\widetilde H^*(\sK_J)$ with a shift in dimension. The first
formula is obtained by summation.
\end{proof}

\begin{remark}
We also obtain that
\[
  \Tor^{-i}_{\Z[v_1,\ldots,v_m]}(\Z[\sK],\Z)\cong\bigoplus_{J\subset[m]}
  \widetilde H^{|J|-i-1}(\sK_J),
\]
which is known in combinatorial commutative algebra as the
\emph{Hochster formula}.

The multiplication in $H^*(\zk)$ may be also described in terms of
full subcomplexes of~$\sK$: the product of $\alpha\in
H^{-i,2J}(\zk)$ and $\beta\in H^{-k,2L}(\zk)$ is zero if $J\cap
L\ne\varnothing$, and otherwise $\alpha\cdot\beta$ is given by a
certain element in $H^{|J|+|L|-i-k-1}(\sK_{J\sqcup L})$,
see~\cite[\S5.1]{pano08}.
\end{remark}

\begin{corollary}
We have
\[
  H^k(\zp)\cong\bigoplus_{J\subset[m]}
  \widetilde H^{k-|J|-1}(P_J)
  \quad\text{and}\quad
  H^{-i,2J}(\zp)\cong\widetilde
  H^{|J|-i-1}(P_J),
\]
where $P_J=\bigcup_{j\in J}F_j\subset P$.
\end{corollary}
\begin{proof}
By considering the barycentric subdivision of $\sK=\sK_P$ (which
is also the barycentric subdivision of $\partial P$) we observe
that the union of facets $\bigcup_{j\in J}F_j$ retracts
onto~$\sK_J$.
\end{proof}

\section{Some homotopy and diffeomorphism types}
We start with an example of the cohomology ring calculation using
Theorem~\ref{zkcoh}.

\begin{example}
1. Let $P$ be a pentagon. Then $\dim\zp=7$ and
\[
  \Z[\sK_P]=\Z[v_1,\ldots,v_5]/(v_iv_j\colon j-i=2\mod 5).
\]
We have the following nontrivial cohomology groups
\begin{align*}
H^0(\zp)&\cong\Z,&&
\text{generated by }1\in R^*(\sK_P)\\
H^3(\zp)&\cong\Z^5,&&\text{generated by }[u_iv_j]\in
R^*(\sK_P)\text{ for }j-i=2\mod5\\
H^4(\zp)&\cong\Z^5,&&\text{generated by }[u_iu_jv_k]\in
R^*(\sK_P)\text{ for }k-i=3,\;k-j=2\mod5\\
H^7(\zp)&\cong\Z,&& \text{generated by }[u_1u_2u_3v_4v_5]\in
R^*(\sK_P).
\end{align*}
The multiplication is also easily determined (e.g.,
$[u_2v_4]\cdot[u_2u_3v_5]=0$ and
$[u_2u_3v_5]\cdot[u_4v_1]=[u_1u_2u_3v_4v_5]$), and we obtain the
following isomorphism of rings:
\[
  H^*(\zp)\cong H^*\bigl((S^3\times S^4)^{\#5}\bigr),
\]
where $M^{\#k}$ denotes the connected sum of $k$ copies of the
manifold~$M$.

2. A similar calculation shows that if $P$ is an $m$-gon, then
\begin{equation}\label{zpcs}
  H^*(\zp)\cong H^*\Bigl(\mathop{\#}_{k=3}^{m-1}
  \bigl(S^k\times S^{m+2-k}\bigr)^{\#(k-2)\binom{m-2}{k-1}}\Bigr).
\end{equation}
\end{example}

In fact, the cohomology ring isomorphism of~\eqref{zpcs} is
induced by a diffeomorphism, so that the moment-angle manifolds
corresponding to polygons are connected sums of sphere products,
with 2 spheres in each product. This description of the
diffeomorphism type of $\zp$ admits the following generalisation
to a series of higher-dimensional polytopes.

Let $P$ be given by~\eqref{ptope} and let $\mb v\in P$ be a
vertex. Choose a hyperplane $\{\mb x\colon (\mb a,\mb x)+b=0\}$
separating $\mb v$ from the other vertices of $P$, i.e. $(\mb
a,\mb v)+b<0$ and $(\mb a,\mb v')+b>0$ for any other vertex $\mb
v'\in P$. We refer to the polytope $P'$ obtained by adding the
inequality $(\mb a,\mb x)+b\ge0$ to~\eqref{ptope} as a
\emph{vertex cut} of~$P$.

\begin{theorem}[{essentially McGavran, see~\cite[Th.~6.3]{bo-me06}}]
Let $P$ be a polytope obtained from a simplex $\Delta^n$ by
applying vertex cut operation $m-n-1$ times. Then $\zp$ is
diffeomorphic to the following connected sum of sphere products:
\[
  \mathop{\#}_{k=3}^{m-n+1}
  \bigl(S^k\times S^{m+n-k}\bigr)^{\#(k-2)\binom{m-n}{k-1}}.
\]
\end{theorem}

For $n=2$ we obtain the diffeomorphism behind cohomology
isomorphism~\eqref{zpcs}. There are also other polytopes $P$ for
which $\zp$ is diffeomorphic to a connected sum of sphere
products, see~\cite{gi-lo09}. However, in general the topology of
$\zp$ is much more complicated, as is shown by the next example.

\begin{example}[{Baskakov~\cite{bask03}, see
also~\cite[\S5.3]{pano08}}] Let $P$ be a 3-dimensional polytope
obtained from the cube by cutting off two non complanar edges.
(The edge cut operation is defined similarly to the vertex cut, by
choosing a hyperplane separating the edge from the other vertices
of the polytope.) The corresponding $\zp$ is an 11-dimensional
manifold. It has a nontrivial \emph{triple Massey product} of
3-dimensional cohomology classes. It follows that $\zp$ is not
\emph{formal} in the sense of the rational homotopy theory; in
particular, $\zp$ cannot be diffeomorphic to a connected sum of
sphere products.
\end{example}

Now let us consider some non polytopal examples.

\begin{example}
Let $\sK$ consist of $m$ disjoint points. Then $\zk$ is the space
of Example~\ref{complexam}.2 and it is homotopy equivalent to the
complement $U(\sK)$ to all codimension-two coordinate planes
in~$\C^m$, see Example~\ref{complexam}. We have
\[
  \Z[\sK]=\Z[v_1,\ldots,v_m]/(v_iv_j\colon 1\le i<j\le m).
\]
The subspace of $(k+1)$-dimensional cocycles in $R^*(\sK)$ is
generated by the monomials
\[
  u_{i_1}u_{i_2}\cdots u_{i_{k-1}}v_{i_k},\quad k\ge2\text{ and }
  i_p\ne i_q\text{ for }p\ne q,
\]
and has dimension $m\binom{m-1}{k-1}$. The subspace of
coboundaries is generated by the elements of the form
\[
  d(u_{i_1}\cdots u_{i_k})
\]
and is $\binom mk$-dimensional. Therefore
\[
\begin{array}{l}
  \dim H^{0}(\zk)=1,\\[2mm]
  \dim H^{1}(\zk)=H^{2}(\zk)=0,\\[2mm]
  \dim H^{k+1}(\zk)=
  m\binom{m-1}{k-1}-\binom mk=(k-1)\binom mk,
  \quad 2\le k\le m,
\end{array}
\]
and multiplication in the cohomology of $\zk$ is trivial. We
therefore have an isomorphism of rings
\begin{equation}\label{ukwed}
  H^*(\zk)\cong
  H^*\Bigl(\bigvee_{k=2}^m
  \bigl(S^{k+1}\bigr)^{\vee(k-1)\binom mk}\Bigr),
\end{equation}
where $X^{\vee k}$ denotes the wedge of $k$ copies of the
space~$X$. In fact, this cohomology ring isomorphism is induced by
a homotopy equivalence:
\end{example}

\begin{theorem}[{Grbi\'c--Theriault \cite[Cor.~9.5]{gr-th07}}]
Let $\sK$ be the $i$-dimensional skeleton of a simplex
$\Delta^{m-1}$, so that $U(\sK)$ is the complement to all
codimension $(i+2)$ coordinate planes in~$\C^m$. Then $U(\sK)$ has
the homotopy type of the wedge of spheres
\[
  \bigvee_{k=i+2}^{m}\bigl(S^{i+k+1}\bigr)^{\vee\binom mk
  \binom{k-1}{i+1}}.
\]
\end{theorem}
For $i=0$ we obtain the homotopy equivalence behind cohomology
isomorphism~\eqref{ukwed}. It is also shown in~\cite{gr-th07} that
$U(\sK)$ has the homotopy type of a wedge of spheres for a wider
class of simplicial complexes, including~\emph{shifted complexes}.

\begin{example}
Let $P$ be a polytope obtained from $\Delta^n$ by applying the
vertex cut operation $p-1$ times, and let $\sK$ consist of $p$
disjoint points. Then the numbers of spheres and their dimensions
in the connected sum $\zp$ correspond to the numbers of spheres
and their dimensions in the wedge $\zk$:
\[
  \zp\cong\mathop{\#}_{k=3}^{p+1}
  \bigl(S^k\times S^{p+2n-k}\bigr)^{\#(k-2)\binom{p}{k-1}},\quad
  \zk\cong\bigvee_{k=3}^{p+1}\bigl(S^k\bigr)^{\vee(k-2)\binom p{k-1}}.
\]
For instance, for $n=3$ and $p=4$ we get
\[
\begin{aligned}
  \zp&\cong\bigl(S^3\times S^7\bigr)^{\#6}\#\bigl(S^4\times S^6\bigr)^{\#8}
  \#\bigl(S^5\times S^5\bigr)^{\#3},\\
  \zk&\cong\bigl(S^3\bigr)^{\vee6}\vee\bigl(S^4\bigr)^{\vee8}
  \vee\bigl(S^5\bigr)^{\vee3}.
\end{aligned}
\]
The nature of this correspondence is yet to be fully understood.
\end{example}

\newpage

\part{Complex-analytic structures on moment-angle manifolds}
Here we review a construction of Bosio--Meersseman~\cite{bo-me06},
which endows an even-dimensional moment-angle manifold $\zp$ with
a non K\"ahler complex-analytic structure of an
\emph{LVM-manifold}. We finish by obtaining some new information
about the Dolbeault cohomology and Hodge numbers of these complex
structures on~$\zp$.

\section{LVM-manifolds}
Let $P$ be a simple polytope~\eqref{ptope} and $\zp$ the
corresponding moment-angle manifold. For simplicity we shall
identify $\zp$ with its embedding $i_Z(\zp)$ in~$\C^m$. As
detailed in Lecture~I, we may describe $\zp$ as an intersection of
$m-n-1$ homogeneous real quadrics with a unit sphere in~$\C^m$:
\[
  \zp=\left\{\begin{array}{ll}
  \mb z\in\C^m\colon&\sum_{k=1}^m\mb g_k|z_k|^2=\textbf 0,\\[2mm]
  &\sum_{k=1}^m|z_k|^2=1,
  \end{array}\right\}
\]
where $\mb g_1,\ldots,\mb g_m$ is a set of vectors in $\R^{m-n-1}$
satisfying conditions~(i) and~(ii) of Section~\ref{fqtp}.

Assume that $m-n-1$ is even an let $m-n-1=2s$. The transpose of
the $2s\times m$ matrix $\varGamma^\star=(\mb g_1,\ldots,\mb
g_m)=(\gamma_{jk})$ defines an inclusion of a $2s$-dimensional
subspace in~$\R^m$, which we denote~$V$. We choose a complex
$s\times m$ matrix $\varOmega=(\omega_{jk})$ such that the image
of the $\R$-linear map
$\C^s\stackrel{\varOmega^t}\llra\C^m\stackrel{\mathrm{Re}}\llra\R^m$
is exactly~$V$. Let $\omega_k\in\C^s$ denote the $k$th column of
$\varOmega$, so that $\varOmega=(\omega_1,\ldots,\omega_m)$. A
sample choice of $\varOmega$ is as follows:
$\omega_{jk}=\gamma_{2j-1,k}+i\gamma_{2j,k}$ for $1\le j\le s$ and
$1\le k\le m$.

We may now use the complex vectors $\omega_k\in\C^s$ instead of
the real vectors $\mb g_k\in\R^{2s}$ in the presentation of $\zp$:
\begin{equation}\label{zpomega}
  \zp=\left\{\begin{array}{ll}
  \mb z\in\C^m\colon&\sum_{k=1}^m\omega_k|z_k|^2=\textbf 0,\\[2mm]
  &\sum_{k=1}^m|z_k|^2=1,
  \end{array}\right\}
\end{equation}

Now define the manifold $N$ as the projectivisation of the
intersection of homogeneous quadrics in~\eqref{zpomega}:
\begin{equation}\label{ndef}
N:=\bigl\{\mb z\in\C
P^{m-1}\colon\omega_1|z_1|^2+\ldots+\omega_m|z_m|^2=\textbf
0\bigr\},\quad\omega_k\in\C^s.
\end{equation}

We therefore have a principal $S^1$-bundle $\zp\to N$.

\begin{theorem}[\cite{meer00}]\label{thlvm}
The manifold $N$ has a holomorphic atlas describing it as a
complex manifold of complex dimension $m-1-s$.
\end{theorem}
\begin{proof}[Sketch of proof]
Consider a holomorphic action of $\C^s$ on $\C^m$ given by
\begin{equation}\label{csact}
\begin{aligned}
  \C^s\times\C^m&\longrightarrow\C^m\\
  (\mb w,\mb z)&\mapsto \bigl(z_1e^{\langle\omega_1,\mb
  w\rangle},\ldots,z_me^{\langle\omega_m,\mb w\rangle}\bigr),
\end{aligned}
\end{equation}
where $\mb w=(w_1,\ldots,w_s)\in\C^s$, and $\langle\omega_k,\mb
w\rangle=\omega_{1k}w_1+\ldots+\omega_{sk}w_s$.

An argument similar to that of the proof of Lemma~\ref{freeact}
shows that the restriction of the action~\eqref{csact} to
$U(\sK_P)\subset\C^m$ is free. (Here $U(\sK_P)$ is the complement
of the coordinate subspace arrangement determined by~$\sK_P$, see
Section~\ref{csac}.) Using a holomorphic atlas transverse to the
orbits of the free action of $\C^s$ on the complex manifold
$U(\sK_P)$ we obtain that the quotient $U(\sK_P)/\C^s$ has a
structure of a complex manifold.

On the other hand, it may be shown that the function
$|z_1|^2+\ldots+|z_m|^2$ on $\C^m$ has a unique minimum when
restricted to an orbit of the free action of $\C^s$ on $U(\sK_P)$.
The set of these minima can be described as
\[
  T:=\bigl\{\mb z\in\C^m\setminus\!\!\{\mathbf 0\}
  \colon\;\omega_1|z_1|^2+\ldots+\omega_m|z_m|^2=\textbf 0\bigr\}.
\]
I follows that the quotient $U(\sK_P)/\C^s$ may be identified
with~$T$, and therefore $T$ acquires a structure of a complex
manifold of dimension $m-s$.

This construction may be projectivised by replacing $\C^m$ by $\C
P^{m-1}$ and $U(\sK_P)$ by the complement to an arrangement in $\C
P^{m-1}$. Therefore, $N$ also becomes a complex manifold.
\end{proof}

The manifold $N$ endowed with the complex structure of
Theorem~\ref{thlvm} is referred to as an \emph{LVM-manifold}.
These manifolds were described by Meersseman~\cite{meer00} as a
generalisation of the construction of Lopez de
Medrano--Verjovsky~\cite{lo-ve97}.

\begin{remark}
The embedding of $T$ in $\C^m$ and of $N$ in $\C P^{m-1}$ given
by~\eqref{ndef} is not holomorphic.
\end{remark}

\section{$\zp$ as an LVM-manifold}
By using the principal $S^1$-bundle $\zp\to N$ and playing with
redundant inequalities one may also endow $\zp$ (if its dimension
$m+n$ is even) or $\zp\times S^1$ (if $m+n$ is odd) with a
structure of an LVM-manifold. We first summarise the effects that
a redundant inequality in~\eqref{ptope} has on different spaces
appeared above.

\begin{proposition}
The following conditions are equivalent:
\begin{itemize}
\item[(a)] $\langle\mb a_i,\mb x\:\rangle+b_i\ge0$ is a redundant
inequality in~\eqref{ptope} (i.e. $F_i=\varnothing$);
\item[(b)] $\zp\subset\{\mb z\in\C^m\colon z_i\ne0\}$;
\item[(c)] $U(\sK_P)$ has a factor $\C^\times$ on the $i$th
coordinate;
\item[(d)] $\mathbf 0\notin\conv(\omega_k\colon k\ne i)$.
\end{itemize}
\end{proposition}
\begin{proof}
The equivalence (a)$\Leftrightarrow$(b)$\Leftrightarrow$(c)
follows directly from the definition of~$\zp$ and~$U(\sK_P)$. The
equivalence (a)$\Leftrightarrow$(d) follows from
Lemma~\ref{propgd}.
\end{proof}

\begin{theorem}[{\cite[Th.~12.2]{bo-me06}}]\label{bometh}
Let $\zp$ be the moment angle manifold corresponding to an
$n$-dimensional simple polytope~\eqref{ptope} defined by $m$
inequalities.
\begin{itemize}
\item[(a)] If $m+n$ is even then $\zp$ has a complex structure as
an LVM-manifold.
\item[(b)] If $m+n$ is odd then $\zp\times S^1$ has a complex structure as
an LVM-manifold.
\end{itemize}
\end{theorem}
\begin{proof}
(a) Since $m+n$ is even, $m-n-1$ is odd. We add one redundant
inequality $1\ge0$ to~\eqref{ptope}, and denote the resulting
moment-angle manifold by $\zp'$. We have $\zp'\cong\zp\times S^1$,
and it follows from Construction~\ref{2ndme} that $\zp'$ is given
by
\[
  \left\{\begin{array}{lrcccrcr}
  \mb z\in\C^{m+1}\colon&\mb g_1|z_1|^2&+&\ldots&+&\mb g_m|z_m|^2&&=\mathbf 0,\\[1mm]
  &|z_1|^2&+&\ldots&+&|z_m|^2&-&|z_{m+1}|^2=0,\\[1mm]
  &|z_1|^2&+&\ldots&+&|z_m|^2&+&|z_{m+1}|^2=1,
  \end{array}\right\}
\]
where $\varGamma^\star=(\mb g_1\ldots\:\mb g_m)$ is the
$(m-n-1)\times m$ matrix of coefficients of the homogeneous
quadrics for~$\zp$. The corresponding matrix for $\zp'$ is
therefore
\[
  {\varGamma^\star}^\prime=\begin{pmatrix}\mb g_1&\ldots&\mb
  g_m&0\\1&\ldots&1&-1
  \end{pmatrix}.
\]
Its height is $m-n$ and therefore even, so that we may replace it
by an $s\times(m+1)$ complex matrix
$\varOmega=(\omega_1\ldots\:\omega_{m+1})$ where $m-n=2s$, and
define
\begin{equation}\label{nprimeeven}
  N':=\bigl\{\mb z\in\C P^m\colon
  \omega_1|z_1|^2+\ldots+\omega_{m+1}|z_{m+1}|^2=\textbf
  0\bigr\}.
\end{equation}
Then $N'$ has a complex structure as an LVM-manifold by
Theorem~\ref{thlvm}. On the other hand,
\[
  N'\cong\zp'/S^1=(\zp\times S^1)/S^1\cong\zp,
\]
and $\zp$ also has a complex structure.

\smallskip

(b) The proof here is similar, but we have to add two redundant
inequalities to~\eqref{ptope}. Then $\zp'\cong\zp\times S^1\times
S^1$ is given by
\[
  \left\{\begin{array}{lrcrclcr}
  \mb z\in\C^{m+2}\colon&\mb g_1|z_1|^2&+\;\ldots\;+&\mb g_m|z_m|^2&&&&=\mathbf 0,\\[1mm]
  &|z_1|^2&+\;\ldots\;+&|z_m|^2&-&|z_{m+1}|^2&&=0,\\[1mm]
  &|z_1|^2&+\;\ldots\;+&|z_m|^2&&&-&|z_{m+2}|^2=0,\\[1mm]
  &|z_1|^2&+\;\ldots\;+&|z_m|^2&+&|z_{m+1}|^2&+&|z_{m+2}|^2=1.
  \end{array}\right\}
\]
The matrix of coefficients of the homogeneous quadrics is
therefore
\[
  {\varGamma^\star}^\prime=\begin{pmatrix}\mb g_1&\ldots&\mb
  g_m&0&0\\1&\ldots&1&-1&0\\1&\ldots&1&0&-1
  \end{pmatrix}.
\]
We replace it by an $s\times(m+2)$ complex matrix
$\varOmega=(\omega_1\ldots\:\omega_{m+2})$ where $m-n+1=2s$, and
define
\begin{equation}\label{nprimeodd}
  N':=\bigl\{\mb z\in\C P^{m+1}\colon
  \omega_1|z_1|^2+\ldots+\omega_{m+2}|z_{m+2}|^2=\textbf
  0\bigr\}.
\end{equation}
Then $N'$ has a complex structure as an LVM-manifold by
Theorem~\ref{thlvm}. On the other hand,
\[
  N'\cong\zp'/S^1=(\zp\times S^1\times S^1)/S^1\cong\zp\times S^1,
\]
and $\zp\times S^1$ also has a complex structure.
\end{proof}

We demonstrate this construction on the two classical examples of
non K\"ahler complex manifolds.

\begin{example}[Hopf manifold]\label{hopf}
An example of a non K\"ahler compact complex manifold is provided
by the quotient of $\C^m\setminus\!\{\mathbf 0\}$ by the action of
$\Z$ generated by the coordinatewise multiplication by a complex
number $\tau$ such that $|\tau|\ne1$. The complex manifolds
obtained in this way are known as the \emph{Hopf manifolds}; they
are all diffeomorphic to $S^{2m-1}\times S^1$. If $m=1$ then the
Hopf manifold is a complex torus; otherwise it is not K\"ahler as
its second cohomology group is zero. The Hopf manifolds may be
obtained as particular cases of moment-angle manifolds with the
complex structures described above.

\hangindent=37mm \hangafter=11 Let $P$ be an $n$-simplex, so that
$m=n+1$, $\zp\cong S^{2n+1}$ is given by a single equation
$|z_1|^2+\ldots+|z_m|^2=1$ in $\C^m$, and $\varGamma^\star$ is
empty. Since $m+n$ is odd, we need to consider $\zp'\cong\zp\times
S^1\times S^1$ given by
\[
  \left\{\begin{array}{lrcrclcr}
  \mb z\in\C^{m+2}\colon
  &|z_1|^2&+\;\ldots\;+&|z_m|^2&-&|z_{m+1}|^2&&=0,\\[1mm]
  &|z_1|^2&+\;\ldots\;+&|z_m|^2&&&-&|z_{m+2}|^2=0,\\[1mm]
  &|z_1|^2&+\;\ldots\;+&|z_m|^2&+&|z_{m+1}|^2&+&|z_{m+2}|^2=1,
  \end{array}\right\}
\]
Then we replace the $2\times m+2$ matrix
\[
  {\varGamma^\star}^\prime=\begin{pmatrix}1&1&\ldots&1&-1&0\\1&1&\ldots&1&0&-1
  \end{pmatrix}.
\]
by the $1\times(m+2)$ complex matrix
$\varOmega=(\omega_1\ldots\:\omega_{m+2})$ where $\omega_k=1+i$
for $1\le k\le m$,\; $\omega_{m+1}=-1$, and $\omega_{m+2}=-i$. The
corresponding configuration of points in~$\C$ is shown on the
figure; note that it satisfies conditions~(i) and~(ii) of
Section~\ref{fqtp}. Then the manifold $N'$ defined
by~\eqref{nprimeodd} acquires a structure of a complex manifold of
dimension~$m$, and we have $N'\cong\zp\times S^1$. We therefore
obtain a complex structure on $S^{2n+1}\times S^1$; this complex
structure may be shown to be equivalent to that of a Hopf
manifold.

\noindent\raisebox{0.5\baselineskip}[0pt]
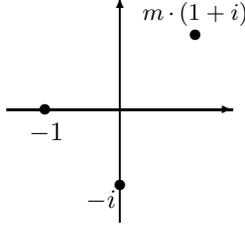
{%
\begin{picture}(30,30)
  \put(25,25){\circle*{1.3}}
  \put(5,15){\circle*{1.3}}
  \put(15,5){\circle*{1.3}}
  \put(18,27){\small$m\,{\cdot}\,(1+i)$}
  \put(3,11){$-1$}
  \put(10.5,2){$-i$}
  \put(15,0){\vector(0,1){30}}
  \put(0,15){\vector(1,0){30}}
\end{picture}%
}
\vspace{-0.9\baselineskip}
\end{example}

\begin{example}[Calabi--Eckmann manifold]
Another example of non K\"ahler complex manifold is due to
Calabi--Eckmann. It is obtained by endowing the fibre $S^1\times
S^1$ of the product of two Hopf bundles $S^{2p+1}\times
S^{2q+1}\to\C P^p\times \C P^q$ with a structure of a complex
torus; therefore the total space $S^{2p+1}\times S^{2q+1}$ also
acquires a complex structure. Like in the case of Hopf manifolds,
these complex structures are particular cases of the complex
structures on $\zp$ described in Theorem~\ref{bometh}.

\hangindent=37mm \hangafter=11 Let
$P=\Delta^{p-1}\times\Delta^{q-1}$ be a product of two simplices
($p>1$ and $q>1$), so that $m=p+q=n+2$, and $\zp\cong
S^{2p-1}\times S^{2q-1}$ is given by the equations of
Example~\ref{prodsimex}. Since $m+n$ even, we need to consider
$\zp'\cong\zp\times S^1$ given by
\[
  \left\{\begin{array}{ll}
  \mb z\in\C^{m+1}\colon&|z_1|^2+\ldots+|z_p|^2-|z_{p+1}|^2-\ldots-|z_m|^2=0,\\[1mm]
  &|z_1|^2+\ldots+|z_m|^2-|z_{m+1}^2|=0,\\[1mm]
  &|z_1|^2+\ldots+|z_m|^2+|z_{m+1}^2|=1.
  \end{array}\right\}
\]
Then we replace the $2\times(m+1)$ matrix
\[
  {\varGamma^\star}^\prime=\begin{pmatrix}1&\ldots&1&-1&\ldots&-1&0\\1&\ldots&1&1&\ldots&1&-1
  \end{pmatrix}.
\]
by the $1\times(m+1)$ complex matrix
$\varOmega=(\omega_1\ldots\:\omega_{m+1})$ where $\omega_k=1+i$
for $1\le k\le p$, \; $\omega_k=-1+i$ for $p+1\le k\le m$, and
$\omega_{m+1}=-i$. The corresponding configuration of points
in~$\C$ is shown on the figure; it also satisfies conditions~(i)
and~(ii) of Section~\ref{fqtp}. Then the manifold $N'$ defined
by~\eqref{nprimeeven} acquires a structure of a complex manifold
of dimension~$m-1$, and we have $N'\cong\zp$. We therefore obtain
a complex structure on $S^{2p-1}\times S^{2q-1}$; this complex
structure may be shown to be equivalent to that of a
Calabi--Eckmann manifold.

\noindent\raisebox{0.5\baselineskip}[0pt]
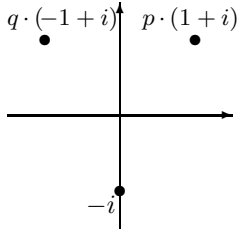
{%
\begin{picture}(30,30)
  \put(25,25){\circle*{1.3}}
  \put(5,25){\circle*{1.3}}
  \put(15,5){\circle*{1.3}}
  \put(18,27){\small$p\,{\cdot}\,(1+i)$}
  \put(0,27){\small$q\,{\cdot}\,(\!-1+i)$}
  \put(10.5,2){$-i$}
  \put(15,0){\vector(0,1){30}}
  \put(0,15){\vector(1,0){30}}
\end{picture}%
}
\vspace{-0.9\baselineskip}
\end{example}

\begin{remark}
There is also a more direct method of giving $\zp$ a complex
structure, without referring to projectivised quadrics and LVM
manifolds, see~\cite{pa-us10}. It identifies $\zp$ with the
quotient of $U(\sK_P)$ by a holomorphic action of $\C^\ell$, and
may be generalised to some non polytopal moment-angle manifolds
$\zk$ (namely those for which $\sK$ is the underlying complex of a
complete simplicial fan).
\end{remark}

\section{Dolbeault cohomology and Hodge numbers}
Here we use the construction~\cite{me-ve04} of holomorphic
principal bundles over projective toric varieties and a spectral
sequence due to Borel~\cite{bore} to describe the Dolbeault
cohomology of $\zp$ in the case when $P$ is a Delzant polytope.

Given a complex $n$-dimensional manifold $M$, there is a
decomposition $\Omega_\C^*(M)=\bigoplus\Omega^{p,q}(M)$ of the
space of complex differential forms on $M$ into a direct sum of
the subspaces of \emph{$(p,q)$-forms} for $0\le p,q\le n$, and the
\emph{Dolbeault differential}
$\bar\partial\colon\Omega^{p,q}(M)\to\Omega^{p,q+1}(M)$. The
dimensions $h^{p,q}$ of the Dolbeault cohomology groups
$H_{\bar\partial}^{p,q}(M)$ are known as the \emph{Hodge numbers}
of~$M$. They are important invariants of the complex structure
of~$M$.

Assume now that $P$ given by~\eqref{ptope} is a Delzant polytope
(see Section~\ref{levelset}). Then there is a principal
$T^{m-n}$-bundle $\zp\to X_P$ (see Lemma~\ref{freeact}), where
$X_P$ is the nonsingular projective toric variety corresponding
to~$P$.

Assume now that $m-n$ is even (otherwise we add one redundant
inequality to~\eqref{ptope}), and let $m-n=2\ell$. A construction
of~\cite{me-ve04} defines a structure of a holomorphic principal
bundle on $\zp\to X_P$, with fibre a compact complex
$\ell$-dimensional torus~$T_\C^\ell$. A spectral sequence of
Borel~\cite{bore} enables us to calculate the Dolbeault cohomology
of the total space of a holomorphic bundle with a K\"ahlerian
fibre in terms of the Dolbeault cohomology of the fibre and base.
In the case of the bundle $\zp\to X_P$ the Dolbeault cohomology of
the fibre and base are well-known and easily described.

The Dolbeault cohomology of $T_\C^\ell$ is isomorphic to an
exterior algebra on $2\ell$ generators:
\begin{equation}\label{dolbtorus}
  H_{\bar\partial}^{*,*}(T_\C^\ell)\cong
  \Lambda[\xi_1,\ldots,\xi_\ell,\eta_1,\ldots,\eta_\ell],
\end{equation}
where $\xi_i\in H_{\bar\partial}^{1,0}(T_\C^\ell)$ are classes of
holomorphic 1-forms, and $\eta_i\in
H_{\bar\partial}^{0,1}(T_\C^\ell)$ are classes of antiholomorphic
1-forms, for $1\le i\le\ell$. In particular, the Hodge numbers are
given by $h^{p,q}(T_\C^\ell)=\binom\ell p\binom\ell q$.

The Dolbeault cohomology of a nonsingular projective toric variety
$X_P$ is given by a result of Danilov--Jurkiewicz~\cite{dani78}:
\begin{equation}\label{dolbtoric}
  H_{\bar\partial}^{*,*}(X_P)\cong
  \C[v_1,\ldots,v_m]/(\mathcal I_{\sK_P}+\mathcal J_{\Sigma_P}),
\end{equation}
where $v_i\in H_{\bar\partial}^{1,1}(X_P)$ for $1\le i\le m$,
$\mathcal I_{\sK_P}=(v^I\colon I\notin\sK_P)$ is the
Stanley--Reisner ideal (see Definition~\ref{deffr}), and $\mathcal
J_{\Sigma_P}$ is generated by the linear combinations
$\sum_{k=1}^ma_{kj}v_k$ for $1\le j\le n$, where $a_{kj}$ is the
$j$th coordinate of~$\mb a_k$. We have $h^{p,p}(X_P)=h_p(P)$,
where $(h_0(P),h_1(P),\ldots,h_n(P))$ is the \emph{$h$-vector} of
$P$~\cite[\S1.2]{bu-pa02}, and $h^{p,q}(X_P)=0$ for $p\ne q$.
Since $X_P$ is K\"ahler, its de Rham cohomology algebra (with
coefficients in~$\C$) is obtained from the Dolbeault cohomology by
passing to the total degree.

\begin{theorem}[\cite{pa-us10}]\label{dolbzp}
Let $P$ be a $n$-dimensional Delzant polytope defined by $m$
inequalities~\eqref{ptope} of which at most one is redundant,
and~$m-n=2\ell$. Let $\zp$ be the moment-angle manifold with a
complex structure of an LVM-manifold. Then the Dolbeault
cohomology group $H_{\bar\partial}^{p,q}(\zp)$ is isomorphic to
the $(p,q)$-th cohomology group of the differential bigraded
algebra
\[
  \bigl[\Lambda[\xi_1,\ldots,\xi_\ell,\eta_1,\ldots,\eta_\ell]\otimes
  H_{\bar\partial}^{*,*}(X_P),d\bigr]
\]
whose bigrading is defined by~\eqref{dolbtorus}
and~\eqref{dolbtoric}, and differential $d$ of bidegree $(0,1)$ is
defined on the generators as
\[
  dv_i=d\eta_j=0,\quad d\xi_j=w_j,\quad
  1\le i\le m,\;1\le j\le\ell,
\]
where the $w_j$ are certain linearly independent elements in
$H_{\bar\partial}^{1,1}(X_P)$ whose explicit form depends on the
complex structure of~$\zp$.
\end{theorem}

\begin{corollary}\label{hodge}
Let $\zp$ be as in Theorem~\ref{dolbzp}, and let $k\le1$ be the
number of redundant inequalities in~\eqref{ptope}. Then
\begin{itemize}
\item[(a)] $h^{p,0}(\zp)=0$ for $p>0$;
\item[(b)] $h^{0,q}(\zp)=\binom\ell q$ for $q\ge0$;
\item[(c)] $h^{1,q}(\zp)=(\ell-k)\binom\ell{q-1}$ for $q\ge1$.
\end{itemize}
\end{corollary}

\begin{remark}
Note that $h^{1,0}(\zp)<h^{0,1}(\zp)$, which implies that $\zp$ is
not K\"ahlerian.
\end{remark}

\begin{example}Let $\zp\cong S^1\times S^{2n+1}$ be the Hopf manifold of
Example~\ref{hopf}. The corresponding fan is the normal fan
$\Sigma_P$ of the standard $n$-dimensional simplex $P$ with one
redundant inequality. We have $X_P=\C P^n$, and~\eqref{dolbtoric}
describes its cohomology as the quotient of
$\C[v_1,\ldots,v_{n+2}]$ by the two ideals: $\mathcal I$ generated
by $v_1$ and $v_2\cdots v_{n+2}$, and $\mathcal J$ generated by
$v_2-v_{n+2},\ldots,v_{n+1}-v_{n+2}$. The differential algebra of
Theorem~\ref{dolbzp} is therefore given by
$\bigl[\Lambda[\xi,\eta]\otimes\C[t]/t^{n+1},d\bigr]$, and
$dt=d\eta=0$, $d\xi=\alpha t$ for some $\alpha\ne0$. The
nontrivial cohomology classes are represented by the cocycles $1$,
$\eta$, $\xi t^n$ and $\xi\eta t^n$, which gives the following
nonzero Hodge numbers of~$\zp$:
$h^{0,0}=h^{0,1}=h^{n+1,n}=h^{n+1,n+1}=1$.
\end{example}

It is also interesting to compare Theorem~\ref{dolbzp} with the
following description of the ordinary cohomology of~$\zp$.

\begin{theorem}[{\cite[Th.~7.36]{bu-pa02}}]\label{cohomzpred}
The cohomology algebra of $\zp$ (with coefficients in a field) is
isomorphic to the cohomology of the differential graded algebra
\[
  \bigl[\Lambda[u_1,\ldots,u_{m-n}]\otimes
  H^*(X_P),d\bigr],
\]
where $\deg u_i=1$, $\deg v_i=2$, and differential $d$ is defined
on the generators as
\[
  dv_j=0,\quad du_j=\gamma_{j1}v_1+\ldots+\gamma_{jm}v_m,\quad
  1\le j\le m,
\]
where $\varGamma=(\gamma_{jk})$ is given by~\eqref{iprn}.
\end{theorem}

Theorem~\ref{cohomzpred} may be deduced from the general
Theorem~\ref{zkcoh} using homological methods. In the case when
$P$ is Delzant the above theorem, like Theorem~\ref{dolbzp}, is
the collapse result for a spectral sequence (this time the
Leray--Serre spectral sequence of the principal $T^{m-n}$-bundle
$\zp\to X_P$).

\end{document}